\documentclass[a4paper,oneside,11pt]{article}

\usepackage[a4paper]{geometry}
\usepackage{aeguill}                 % ... or a4paper or a5paper or ... 
\usepackage{graphicx}
\usepackage{amsmath,amsfonts,amssymb,amsthm}
\usepackage{color}
\usepackage[utf8]{inputenc} 
\usepackage[T1]{fontenc} 
 
\usepackage[english]{babel} 

\title{Spectral theorems for random walks on mapping class groups and $\text{Out}(F_N)$}
\author{François Dahmani and Camille Horbez}

\usepackage[all]{xy}
\usepackage{bbm}
\begin{document}
\maketitle
%\tableofcontents
\newtheorem{de}{Definition} [section]
\newtheorem{theo}[de]{Theorem} 
\newtheorem{prop}[de]{Proposition}
\newtheorem{lemma}[de]{Lemma}
\newtheorem{cor}[de]{Corollary}
\newtheorem{propd}[de]{Proposition-Definition}

\theoremstyle{remark}
\newtheorem{rk}[de]{Remark}
\newtheorem{ex}[de]{Example}
\newtheorem{question}[de]{Question}

\normalsize

\begin{abstract}
We establish spectral theorems for random walks on mapping class groups of connected, closed, oriented, hyperbolic surfaces, and on $\text{Out}(F_N)$. In both cases, we relate the asymptotics of the stretching factor of the diffeomorphism/automorphism obtained at time $n$ of the random walk to the Lyapunov exponent of the walk, which gives the typical growth rate of the length of a curve -- or of a conjugacy class in $F_N$ -- under a random product of diffeomorphisms/automorphisms. 

In the mapping class group case, we first observe that the drift of the random walk in the curve complex is also equal to the linear growth rate of the translation lengths in this complex. By using a contraction property of typical Teichmüller geodesics, we then lift the above fact to the realization of the random walk on the Teichmüller space. For the case of $\text{Out}(F_N)$, we follow the same procedure with the free factor complex in place of the curve complex, and the outer space in place of the Teichmüller space. A general criterion is given for making the lifting argument possible.
\end{abstract}
\setcounter{tocdepth}{1}
\tableofcontents

\section*{Introduction}

 In a famous work on random products of matrices \cite{FK60}, Furstenberg and Kesten proved that, under very general conditions, the norm of a random product of $n$ matrices grows as a deterministic exponential function of $n$. This result was later improved by Furstenberg, who established in \cite[Theorems 8.5 and 8.6]{Fur63} that if $(A_i)_{i\in\mathbb{N}}$ is a sequence of matrices chosen uniformly at random among a finite generating set for $SL(N,\mathbb{Z})$, then there exists $\lambda>1$ such that for all vectors $v\in\mathbb{R}^{N}\smallsetminus\{0\}$, almost surely, one has $$\lim_{n\to +\infty}\sqrt[n]{||A_n\dots A_1.v||}=\lambda$$ (and $\log\lambda$ is usually called the \emph{Lyapunov exponent} of the random process). Furstenberg's theorem was a starting point in the study of growth of vectors under random products of matrices, which had many further developments, including the famous multiplicative ergodic theorem of Oseledets \cite{Ose68}. Guivarc'h established in \cite[Théorème 8]{Gui90} a spectral theorem, relating the asymptotics of the top eigenvalue $\lambda(A_n\dots A_1)$ of the matrix $A_n\dots A_1$ to the Lyapunov exponent, showing that almost surely, one has $$\lim_{n\to +\infty}\sqrt[n]{\lambda(A_n\dots A_1)}=\lambda$$ (Guivarc'h actually relates the whole spectrum of the matrix $A_n\dots A_1$ to the set of all Lyapunov exponents of the process).
\\
\\
\indent The goal of the present paper is to establish spectral theorems for the random walks on either the mapping class group $\text{Mod}(S)$ of a closed, connected, oriented, hyperbolic surface $S$ -- i.e. the group of all isotopy classes of orientation-preserving diffeomorphisms of $S$ -- or on the group $\text{Out}(F_N)$ of outer automorphisms of a finitely generated free group. 

\paragraph*{Mapping class groups.} We fix once and for all a hyperbolic metric $\rho$ on $S$. Given a simple closed curve $c$ on $S$, we will denote by $l_{\rho}(c)$ the smallest length with respect to $\rho$ of a curve in the isotopy class of $c$. 
% measured by integration  
%{of the norm of its derivative, given by the metric $\rho$}. 
{An important} theorem of Thurston describes the possible growth rates of the lengths of essential simple closed curves on $S$ under iteration of a single mapping class of the surface: Thurston proved in \cite[Theorem 5]{Thu88} that for all $\Phi\in\text{Mod}(S)$, there is a finite set $1\le\lambda_1<\lambda_2<\dots <\lambda_K$ of algebraic integers such that for any essential simple closed curve $c$ on $S$, there exists $i\in\{1,\dots,K\}$ such that $$\lim_{n\to +\infty}\sqrt[n]{l_{\rho}(\Phi^n(c))}=\lambda_i.$$ In addition, the mapping class $\Phi$ is pseudo-Anosov if and only if $K=1$ and $\lambda_1>1$.

In \cite{Kar12}, Karlsson proved a version of Thurston's theorem for random products of mapping classes, describing the growth of simple closed curves on $S$ under such products. Given a probability measure $\mu$ on $\text{Mod}(S)$, the \emph{(right) random walk} on $(\text{Mod}(S),\mu)$ is the Markov chain whose value at time $n$ is a product $\Phi_n$ of $n$ independent random mapping classes $s_i$, all distributed with respect to the law $\mu$, i.e. $\Phi_n=s_1\dots s_n$. In the following statement, a subgroup of $\text{Mod}(S)$ is \emph{nonelementary} if it contains two pseudo-Anosov diffeomorphisms of the surface that generate a free subgroup of $\text{Mod}(S)$  (equivalently \cite{McP89}, it is not virtually cyclic, and does not virtually preserve the isotopy class of any proper essential subsurface of $S$).

\begin{theo}(Karlsson \cite[Corollary 4]{Kar12})\label{K-intro}
Let $\mu$ be a probability measure on $\text{Mod}(S)$, with finite first moment with respect to the Teichmüller metric, whose support generates a nonelementary subgroup of $\text{Mod}(S)$. Then there exists $\lambda>1$ such that for all essential simple closed curves $c$ on $S$, and almost every sample path $(\Phi_n)_{n\in\mathbb{N}}$ of the random walk on $(\text{Mod}(S),\mu)$, one has $$\lim_{n\to +\infty}\sqrt[n]{l_{\rho}(\Phi_n^{-1}(c))}=\lambda.$$ 
\end{theo}

The real number $\log\lambda>0$, which we call the \emph{Lyapunov exponent} of the random walk on $(\text{Mod}(S),\mu)$, is deterministic: it is equal to the drift of the realization of the random walk on the Teichmüller space of the surface, with respect to the Teichmüller metric.

On the other hand, when $\mu$ has finite support, it is known that the mapping class obtained at time $n$ from the 
   %of the %simple
 random walk on $\text{Mod}(S)$ is pseudo-Anosov with probability tending to $1$ exponentially fast \cite{Riv08,Sis11, Mah12,CM13}. 
Using the Borel--Cantelli Lemma, this implies that almost surely, eventually all mapping classes $\Phi_n$ are pseudo-Anosov (and hence they have a well-defined stretching factor $\lambda(\Phi_n)$). We prove the following spectral theorem for random walks on $\text{Mod}(S)$, which relates the asymptotics of the stretching factors $\lambda(\Phi_n)$ to the Lyapunov exponent of the random walk.

\begin{theo}\label{spectral-intro}
Let $S$ be a closed, connected, oriented, hyperbolic surface, and let $\mu$ be a probability measure on $\text{Mod}(S)$, whose support is finite and generates a semigroup containing two independent pseudo-Anosov elements. 
Then for almost every sample path $(\Phi_n)_{n\in\mathbb{N}}$ of the random walk on $(\text{Mod}(S),\mu)$, we have $$\lim_{n\to +\infty}\sqrt[n]{\lambda(\Phi_n)}=\lambda,$$ where $\log\lambda$ is the Lyapunov exponent of the random walk.
\end{theo}

Finiteness of the support of $\mu$ can be replaced by the assumption that $\mu$ has finite second moment with respect to the Teichmüller metric, see Remarks \ref{rk-abstract} and \ref{rk-mod} below. If $\mu$ is only assumed to have finite first moment with respect to the Teichmüller metric, then the probability that $\Phi_n$ is pseudo-Anosov still converges to $1$ \cite[Theorem 1.4]{MT14} but we do not know whether the speed of convergence still ensures that eventually all   $\Phi_n$  are pseudo-Anosov. 
 However, in this situation, Theorem \ref{spectral-intro} remains valid if one replaces the limit by a limsup.  

We also note that Theorem \ref{spectral-intro} remains valid if $S$ is a compact surface with one boundary component, in which case it may be viewed as a particular case of Theorem \ref{spectral-out}.

At the cost of {an anticipation on}   
 the tools and methods (which are detailed later in this introduction), we would like to insist on the difference in nature between Theorems \ref{K-intro} and \ref{spectral-intro}. On the one hand, the geometric interpretation of the growth described by Theorem \ref{K-intro} is an expression of the drift of the realization of the random walk in the Teichmüller space $\mathcal{T}(S)$, in terms of a certain (random) limiting horofunction. On the other hand, the geometric interpretation of our main result (Theorem \ref{spectral-intro}) is a description of the growth of the translation length $\log\lambda(\Phi_n)$ along the axis of $\Phi_n$ in $\mathcal{T}(S)$. While the distance to the origin in $\mathcal{T}(S)$ is subadditive along the trajectories of the random walk (which allows the use of Kingman ergodic theorem to deduce the existence of the drift), the translation length $\log\lambda(\Phi_n)$ along an axis is not, in fact it can be very different for two very close elements. Therefore, the existence of the limit in Theorem \ref{spectral-intro} is far from obvious. The best straightforward
analogue in a free group of the quantity we want to measure would be the length of a
cyclic reduction of a word. However, as we are measuring all distances in the Teichmüller space $\mathcal{T}(S)$, which is not Gromov hyperbolic, understanding the behaviour of $\log\lambda(\Phi_n)$ is much more difficult in our setting. 

This being said, Karlsson formulated a wish, in \cite[(iii) page 220]{Kar12}, to compare his result ``that random walk trajectories eventually look
pseudo-Anosov from [the growth] perspective'' to the genericity of being
pseudo-Anosov. We believe that our result is a very relevant piece of comparison.  

We would finally like to emphasize that, although the mapping class group $\text{Mod}(S)$ has relations with the linear group $GL(N,\mathbb{Z})$ (for example through the symplectic representation), our main result cannot be deduced from Guivarch's: in fact, Theorem \ref{spectral-intro} also applies to highly non-linear situations, for example it applies to generic random walks on the Torelli group of the surface.

\paragraph*{$\text{Out}(F_N)$.} The same question also makes sense for $\text{Out}(F_N)$. An element $\Phi\in\text{Out}(F_N)$ is \emph{fully irreducible} if no power $\Phi^k$ (with $k\neq 0$) fixes the conjugacy class of a proper free factor of $F_N$. Every fully irreducible outer automorphism $\Phi\in\text{Out}(F_N)$ has a well-defined \emph{stretching factor} $\lambda(\Phi)>0$, which gives the exponential growth rate of all primitive conjugacy classes in $F_N$ under iteration of $\Phi$ (an element of $F_N$ is \emph{primitive} if it belongs to some free basis of $F_N$). The analogue of Thurston's theorem, giving possible growth rates of conjugacy classes of $F_N$ under iteration of a single automorphism of $F_N$, also holds true in this context -- this can be established using train track theory, see \cite{BH92,Lev09}.  

The second author of the present paper proved in \cite[Corollary 5.4]{Hor14-2} an analogue of Karlsson's theorem for  random walks on $\text{Out}(F_N)$, showing that all primitive conjugacy classes in $F_N$ grow exponentially fast along typical sample paths of the random walk, with a deterministic exponential growth rate $\lambda$. Here, the length $||g||$ of a nontrivial element $g\in F_N$ should be understood as the smallest word length of a conjugate of $g$, written in some prescribed free basis of $F_N$. A subgroup of $\text{Out}(F_N)$ is \emph{nonelementary} if it is not virtually cyclic, and does not virtually preserve the conjugacy class of any proper free factor of $F_N$.

\begin{theo}(Horbez \cite[Corollary 5.4]{Hor14-2})\label{intro-out}
Let $\mu$ be a probability measure on $\text{Out}(F_N)$, with finite first moment with respect to the Lipschitz metric on outer space, whose support generates a nonelementary subgroup of $\text{Out}(F_N)$. Then there exists $\lambda>1$ such that for all primitive elements $g\in F_N$, and almost every sample path of the random walk on $(\text{Out}(F_N),\mu)$, we have $$\lim_{n\to +\infty}\sqrt[n]{||\Phi_n^{-1}(g)||}=\lambda.$$ 
\end{theo}

As above, we call $\log\lambda$ the \emph{Lyapunov exponent} of the random walk, which is also equal to the drift of the realization of the walk on Culler--Vogtmann's outer space, with respect to the (asymmetric) Lipschitz metric, i.e. $$\log\lambda=\lim_{n\to +\infty}\frac{1}{n}d_{CV_N}(y_0,\Phi_n.y_0)$$ for almost every sample path $(\Phi_n)_{n\in\mathbb{N}}$ of the random walk on $(\text{Out}(F_N),\mu)$ and any $y_0\in CV_N$. Here again, if $\mu$ has finite support, then the automorphism $\Phi_n$ obtained at time $n$ of the process is fully irreducible with high probability \cite{Riv08,Sis11,CM13},    
 and therefore it has a well-defined stretching factor $\lambda(\Phi_n)$. We prove the following spectral theorem for $\text{Out}(F_N)$.

\begin{theo}\label{spectral-out}
Let $\mu$ be a probability measure on $\text{Out}(F_N)$, whose support is finite and generates a semi-group containing two independent fully irreducible elements of  
$\text{Out}(F_N)$. 
Then for $\mathbb{P}$-a.e. sample path of the random walk on $(\text{Out}(F_N),\mu)$, we have $$\lim_{n\to +\infty}\sqrt[n]{\lambda(\Phi_n^{-1})}=\lambda,$$ where $\log\lambda$ is the Lyapunov exponent of the random walk.
\end{theo}

Contrary to the mapping class group case, the stretching factor of an outer automorphism of $F_N$ can be different from the stretching factor of its inverse. It is therefore important in the above statement to consider $\lambda(\Phi_n^{-1})$ and not $\lambda(\Phi_n)$. 
As in the mapping class group case, finiteness of the support can be replaced by the assumption that $\mu$ has finite second moment with respect to the Lipschitz metric on outer space, see Remarks \ref{rk-abstract} and \ref{rk-out} below. If $\mu$ is only assumed to have finite first moment with respect to the Lipschitz metric, then one should 
 replace the limit by a limsup. 

\paragraph*{Strategy of proofs.} We now explain the general idea of the proofs of Theorems \ref{spectral-intro} and \ref{spectral-out}. 

The mapping class group $\text{Mod}(S)$ acts both on the Teichmüller space $\mathcal{T}(S)$, and on the curve complex $\mathcal{C}(S)$, which is known to be Gromov hyperbolic thanks to work of Masur and Minsky \cite{MM99}.

The behaviour of the realization of the random walk of $\text{Mod}(S)$ on the hyperbolic complex $\mathcal{C}(S)$ is well-understood: there is a rich literature on random walks on hyperbolic spaces \cite{Kai00, CM13,MT14}. Typical sample paths of the realization of the random walk on $\text{Mod}(S)$ converge to a boundary point $\xi\in\partial_{\infty}\mathcal{C}(S)$. The length of the overlap between the geodesic ray from a basepoint $x_0\in\mathcal{C}(S)$ to $\xi$ and a quasi-axis of the mapping class $\Phi_n$ obtained at time $n$ of the process grows linearly with $n$. In particular the translation length of $\Phi_n$ in $\mathcal{C}(S)$ grows linearly with $n$, with speed equal to the drift of the realization on $\mathcal{C}(S)$ of the random walk on $\text{Mod}(S)$.

On the other hand, the logarithm of the stretching factor of $\Phi_n$ is equal to the translation length of $\Phi_n$ on the Teichmüller space $\mathcal{T}(S)$, equipped with the Teichmüller metric. Therefore, we need to establish the same property as above for the realization of the random walk on $\mathcal{T}(S)$: namely, we need to show that the translation length of $\Phi_n$ in $\mathcal{T}(S)$ grows linearly fast at a speed given by the drift   % of the realization
  of the random walk in the Teichmuller metric. 
   %on $\mathcal{T}(S)$. 
The rough idea is that typical geodesics in $\mathcal{T}(S)$ have a hyperbolic-like behaviour. To make this precise, we appeal to a contraction property established by Dowdall, Duchin and Masur in \cite{DDM14}. This asserts that any Teichmüller geodesic segment $I$ whose projection to $\mathcal{C}(S)$ makes progress in $\mathcal{C}(S)$, is contracting in $\mathcal{T}(S)$, in the sense that every other Teichmüller geodesic with same projection as $I$ in $\mathcal{C}(S)$, must pass uniformly close to $I$ in $\mathcal{T}(S)$. We then prove that the realization in $\mathcal{T}(S)$ of a typical sample path of the random walk on $(\text{Mod}(S),\mu)$ stays close to a geodesic ray which contains infinitely many such subsegments. Since the quasi-axis of $\Phi_n$  
 in $\mathcal{C}(S)$ passes close to $x_0$, the contraction property implies  that the translation axis of $\Phi_n$ in $\mathcal{T}(S)$ has to pass close to these subsegments, and hence close to the basepoint $y_0$ (we call this a \emph{lifting argument}).

This implies that the translation length of $\Phi_n$ in $\mathcal{T}(S)$ grows linearly fast, at a speed equal to the drift of the realization of the random walk in $\mathcal{T}(S)$, as required.
\\
\\
\indent Our \emph{lifting argument} is presented in a more abstract setting, see Theorem \ref{spectral-gen} for a precise statement. We require having two actions of a group $G$ on both a metric space $Y$ and a hyperbolic metric space $X$, with a coarsely $G$-equivariant map from $Y$ to $X$, such that geodesics in $Y$ map to unparameterized quasigeodesics in $X$. Typical geodesics in $Y$ are required to contain infinitely many subsegments satisfying the above contraction property. Under such assumptions, if typical elements of $G$ act with a translation axis in $Y$, then the translation length in $Y$ of the element obtained at time $n$ of the random walk on $G$ grows linearly fast, with speed equal to the drift of the realization of the random walk in $Y$.
\\
\\
\indent Similarly, the group $\text{Out}(F_N)$ acts both on Culler--Vogtmann's outer space $CV_N$, and on the free factor complex $FF_N$, whose hyperbolicity was established by Bestvina and Feighn \cite{BF12}. Again, the logarithm of the stretching factor of the automorphism $\Phi_n$ obtained at time $n$ of the process is equal to the translation length of $\Phi_n$ in outer space, equipped with the (asymmetric) Lipschitz metric. We give in Section \ref{sec-contr-out} a similar -- though more technical -- condition ensuring that a folding path in outer space satisfies the above contraction property, and prove that typical folding paths in $CV_N$ contain infinitely many such subsegments. Theorem \ref{spectral-out} then follows from the same general criterion established in Section \ref{abstract}.

\paragraph*{Structure of the paper.} The paper is organized as follows. In Section \ref{sec1}, we review properties of random walks on groups acting on hyperbolic spaces, which mainly come from the work of Calegari--Maher \cite{CM13} and Maher--Tiozzo \cite{MT14}. In Section \ref{abstract}, we state our general criterion for making the \emph{lifting argument} possible. We check in Section \ref{sec3} that typical Teichmüller geodesics in $\mathcal{T}(S)$ contain infinitely many subsegments that have the required contraction property, and we use this to derive Theorem \ref{spectral-intro}. We check in Section \ref{sec4} that typical folding lines in outer space contain infinitely many subsegments having the required contraction property, and we use this to derive Theorem \ref{spectral-out}. 

\subsection*{Acknowledg{e}ments.}

We would like to thank Ursula Hamenstädt for discussions 
  during the \emph{Young Geometric Group Theory IV} meeting held in Spa in January 2015, regarding contraction properties of folding lines in outer space. We also thank the anonymous referees for their careful reading of the paper and their valuable suggestions. The authors acknowledge support from ANR-11-BS01-013, from the Institut Universitaire de France (FD) and from the Lebesgue Center of Mathematics (CH).

\section{Random walks on groups acting on hyperbolic spaces}\label{sec1}

\subsection{Random walks on groups: background and notations}

\paragraph*{General notations.} Given a probability measure $\mu$ on a group $G$, the \emph{(right) random walk} on $G$ with respect to the measure $\mu$ is the Markov chain on $G$ whose initial distribution is the Dirac measure at the identity element, with transition probabilities $p(x,y):=\mu(x^{-1}y)$. The \emph{step space} for the random walk is the product probability space $\Omega:=(G^{\mathbb{N}^{\ast}},\mu^{\otimes\mathbb{N}^{\ast}})$. The position of the random walk at time $n$ is given from its position $g_0=e$ at time $0$ by successive multiplications on the right of independent $\mu$-distributed increments $s_i$, i.e. $g_n=s_1\dots s_n$. We equip the \emph{path space} $\mathcal{P}:=G^{\mathbb{N}}$ with the $\sigma$-algebra generated by the cylinders $\{\textbf{g}\in\mathcal{P}|g_i=g\}$ for all $i\in\mathbb{N}$ and all $g\in G$, and the probability measure $\mathbb{P}$ induced by the map $(s_1,s_2,\dots)\mapsto (g_0,g_1,g_2,\dots)$. Elements of the path space $\mathcal{P}$ are called \emph{sample paths} of the random walk.

We will also need to work with the space $\overline{\mathcal{P}}:=G^\mathbb{Z}$ of \emph{bilateral paths} ${\textbf{g}}:=(g_n)_{n\in\mathbb{Z}}$ corresponding to bilateral sequences of independent $\mu$-distributed increments $(s_n)_{n\in\mathbb{Z}}$ by the formula $g_n=g_{n-1}s_n$, with $g_0=e$. We let $\overline{\Omega}:=(G^{\mathbb{Z}},\mu^{\otimes\mathbb{Z}})$, and we equip $\overline{\mathcal{P}}$ with the probability measure $\overline{\mathbb{P}}$ induced by the map $(s_n)_{n\in\mathbb{Z}}\mapsto (g_n)_{n\in\mathbb{Z}}$. We denote by $\check{\mu}$ the probability measure on $G$ defined by $\check{\mu}(g):=\mu(g^{-1})$ for all $g\in G$, and by $(\check{\mathcal{P}},\check{\mathbb{P}})$ the corresponding path space. Then there is a natural isomorphism between $(\overline{\mathcal{P}},\overline{\mathbb{P}})$ and the product space $(\check{\mathcal{P}},\check{\mathbb{P}})\otimes (\mathcal{P},\mathbb{P})$, mapping a bilateral path $(g_n)_{n\in\mathbb{Z}}$ to the pair of unilateral paths $((g_{-n})_{n\geq 0},(g_n)_{n\ge 0})$ (note that both paths are initialized by $g_0=e$). 
 The \emph{Bernoulli shift} $\sigma$, defined in $\overline{\Omega}$ by $\sigma((s_n)_{n\in\mathbb{Z}}):=(s_{n+1})_{n\in\mathbb{Z}}$, induces an ergodic transformation $U$ on the space of bilateral paths which satisfies $$(U^k.\mathbf{g})_n=g_k^{-1}g_{n+k}$$ for all $k,n\in\mathbb{Z}$.

\paragraph{Moment and drift.} Assume now that $G$ acts by isometries on a metric space $(X,d_X)$, and let $x_0\in X$. We say that a probability measure $\mu$ on $G$ has \emph{finite first moment} with respect to $d_X$ if $$\int_{G}d_X(x_0,g.x_0)d\mu(g)<+\infty.$$ We say that $\mu$ has \emph{finite second moment} with respect to $d_X$ if $$\int_{G}d_X(x_0,g.x_0)^2d\mu(g)<+\infty.$$ It follows from Kingman's subadditive ergodic theorem \cite{Kin68} that if $\mu$ has finite first moment with respect to $d_X$, then for $\mathbb{P}$-a.e. sample path of the random walk on $(G,\mu)$, the limit $$\lim_{n\to +\infty}\frac{1}{n}d_X(x_0,g_n.x_0)$$ exists. It is called the \emph{drift} of the random walk on $(G,\mu)$ with respect to the metric $d_X$.

\subsection{General definitions in coarse geometry}

Given two metric spaces $X$ and $Y$, and a constant $K\in\mathbb{R}$, a \emph{$K$-quasi-isometric embedding} from $X$ to $Y$ is a map $f:X\to Y$ such that for all $x,x'\in X$, one has $$\frac{1}{K}d_X(x,x')-K\le d_{Y}(f(x),f(x'))\le Kd_X(x,x')+K.$$ A \emph{$K$-quasigeodesic} in $X$ is a $K$-quasi-isometric embedding from an interval in $\mathbb{R}$ to $X$. A \emph{$K$-unparameterized quasigeodesic} is a map $\tau':I'\to X$, where $I'\subseteq\mathbb{R}$ is an interval, such that there exists an interval $I\subseteq\mathbb{R}$ and a non-decreasing homeomorphism $\theta:I\to I'$ such that $\tau'\circ\theta$ is a $K$-quasigeodesic.

\subsection{Background on Gromov hyperbolic spaces}

Let $(X,d_X)$ be a metric space, and let $x_0\in X$ be some basepoint. For all $x,y\in X$, the \emph{Gromov product} of $x$ and $y$ with respect to $x_0$ is defined as
\begin{displaymath}
(x|y)_{x_0}:=\frac{1}{2}(d_X(x_0,x)+d_X(x_0,y)-d_X(x,y)).
\end{displaymath}

A metric space $X$ is \emph{Gromov hyperbolic} if there exists a constant $\delta\ge 0$ such that for all $x,y,z,x_0\in X$, one has $$(x|y)_{x_0}\ge\min\{(x|z)_{x_0},(y|z)_{x_0}\}-\delta.$$  When $X$ is geodesic, hyperbolicity of $X$ is equivalent to a \emph{thin triangles} condition: there exists $\delta'\ge 0$ such that for every geodesic triangle $\Delta$ in $X$, any side of $\Delta$ is contained in the $\delta'$-neighborhood of the union of the other two sides. The smallest such $\delta'$ is then called the \emph{hyperbolicity constant} of $X$.

Let $X$ be a Gromov hyperbolic metric space. A sequence $(x_n)_{n\in\mathbb{N}}\in X^{\mathbb{N}}$ \emph{converges to infinity} if the Gromov product $(x_n|x_m)_{x_0}$ goes to $+\infty$ as $n$ and $m$ both go to $+\infty$. Two sequences $(x_n)_{n\in\mathbb{N}}$ and $(y_n)_{n\in\mathbb{N}}$ that both converge to infinity are \emph{equivalent} if the Gromov product $(x_n|y_m)_{x_0}$ goes to $+\infty$ as $n$ and $m$ go to $+\infty$. It follows from the hyperbolicity of $(X,d)$ that this is indeed an equivalence relation. The \emph{Gromov boundary} $\partial_{\infty} X$ is defined to be the collection of equivalence classes of sequences that converge to infinity. 

An isometry $\phi$ of $X$ is \emph{loxodromic} if for all $x\in X$, the orbit map
\begin{displaymath}
\begin{array}{cccc}
\mathbb{Z}&\to & X\\
n &\mapsto &\phi^n.x
\end{array}
\end{displaymath} 
\noindent is a quasi-isometric embedding. Given $K>0$, we say that a line $\gamma:\mathbb{R}\to X$ is a \emph{$K$-quasi-axis} for $\phi$ if $\gamma$ is a $K$-quasigeodesic, and there exists $T\in\mathbb{R}$ such that for all $t\in\mathbb{R}$ and all $k\in\mathbb{Z}$, one has $$d_X(\phi^k.\gamma(t),\gamma(t+kT))\le K.$$ We note that there exists $K$, only depending on the hyperbolicity constant of $X$, such that all loxodromic isometries of $X$ have a $K$-quasi-axis in $X$. Any loxodromic isometry has exactly two fixed points in $\partial_{\infty}X$ and acts with north-south dynamics on $X\cup\partial_{\infty}X$.

\subsection{Random walks on groups acting on Gromov hyperbolic spaces}

In this section, we review work by Calegari--Maher \cite{CM13}, further developed by Maher--Tiozzo \cite{MT14}, concerning random walks on groups acting by isometries on Gromov hyperbolic spaces   
(these extend previous work of Kaimanovich \cite{Kai00} who considered the case of proper actions on proper hyperbolic spaces). We will not always seek for optimal assumptions on the measure $\mu$, we refer to \cite{MT14} for more precise statements, see also Remark \ref{rk-abstract} below.

Let $X$ be a Gromov hyperbolic space, and let $G$ be a countable group acting on $X$ by isometries. 
 A probability measure on $G$ is then called \emph{nonelementary} (for the action on $X$) if the subsemigroup of $G$ generated by its support contains a pair of loxodromic elements with disjoint pairs of fixed points in the Gromov boundary of $X$.

\begin{theo}\textbf{(Convergence to the boundary)}~(Calegari--Maher \cite[Theorem 5.34]{CM13}, Maher--Tiozzo \cite[Theorem 1.1]{MT14})\label{hyp-cv}
Let $G$ be a countable group that acts by isometries on a separable Gromov hyperbolic space $X$, and let $\mu$ be a nonelementary probability measure on $G$. Then for any $x_0\in X$, and $\mathbb{P}$-a.e. sample path $\mathbf{\Phi}:=(\Phi_n)_{n\in\mathbb{N}}$ of the random walk on $(G,\mu)$, the sequence $(\Phi_n.x_0)_{n\in\mathbb{N}}$ converges to a point $\text{bnd}(\mathbf{\Phi})\in\partial_{\infty}X$.
\end{theo}

\begin{theo}\textbf{(Positive drift)}~(Calegari--Maher \cite[Theorem 5.34]{CM13}, Maher--Tiozzo \cite[Theorem 1.2]{MT14})\label{drift}
Let $G$ be a countable group which acts by isometries on a separable Gromov hyperbolic metric space $(X,d_X)$, and let $\mu$ be a nonelementary probability measure on $G$ with finite first moment with respect to $d_X$. Then there exists $L>0$ such that for $\mathbb{P}$-a.e. sample path $(\Phi_n)_{n\in\mathbb{N}}$ of the random walk on $(G,\mu)$, one has $$\lim_{n\to +\infty}\frac{1}{n} d_X(x_0,\Phi_n.x_0)=L.$$
\end{theo}

As recalled above, the existence of the limit in Theorem \ref{drift} is an application of Kingman's subadditive ergodic theorem and only requires that $\mu$ has finite first moment with respect to $d_X$ (without any assumption on $X$, in particular $X$ need not be hyperbolic). The hard part of Theorem \ref{drift} is to prove positivity of $L$, for which hyperbolicity of $X$ and nonelementarity of $\mu$ are crucial.

\begin{theo}\textbf{(Sublinear tracking by quasi-geodesic rays)}~(Maher--Tiozzo \cite[Theorem 1.3]{MT14})\label{souslin}
Let $G$ be a countable group which acts by isometries on a separable Gromov hyperbolic geodesic metric space $(X,d_X)$, and let $\mu$ be a nonelementary probability measure on $G$ with finite first moment with respect to $d_X$. Let $L>0$ be the drift of the random walk on $(G,\mu)$ with respect to $d_X$. Then for $\mathbb{P}$-a.e. sample path $\mathbf{\Phi}:=(\Phi_n)_{n\in\mathbb{N}}$ of the random walk on $(G,\mu)$, there exists a quasi-geodesic ray $\tau:\mathbb{R}_+\to X$ limiting at $\text{bnd}(\mathbf{\Phi})$ such that $$\lim_{n\to +\infty}\frac{1}{n}d_X(\Phi_n.x_0,\tau(Ln))=0.$$ 
\end{theo}

\begin{theo}\textbf{(Random isometries are loxodromic)}~(Calegari--Maher \cite[Theorem 5.35]{CM13}, Maher--Tiozzo \cite[Theorem 1.4]{MT14})\label{loxo}
Let $G$ be a countable group which acts by isometries on a separable Gromov hyperbolic metric space $X$, and let $\mu$ be a nonelementary probability measure on $G$ with finite support. Then for $\mathbb{P}$-a.e. sample path $(\Phi_n)_{n\in\mathbb{N}}$ of the random walk on $(G,\mu)$, there exists $n_0\in\mathbb{N}$ such that for all $n\ge n_0$, the element $\Phi_n$ acts loxodromically on $X$.
\end{theo}

\begin{proof}
Calegari and Maher prove in \cite[Theorem 5.35]{CM13} that the probability that $\Phi_n$ be loxodromic converges exponentially fast to $1$. Theorem \ref{loxo} then follows by applying the Borel--Cantelli lemma. 
\end{proof} 

Given $\kappa>0$, a quasigeodesic segment $\gamma:J\to X$ (where $J\subseteq\mathbb{R}$ is a segment), and a quasigeodesic $\gamma':I\to X$ (where $I\subseteq\mathbb{R}$ is an interval), we say that $\gamma'$ \emph{crosses $\gamma$ up to distance $\kappa$} if there exists an increasing map $\theta:J\to I$ such that $d_X(\gamma(t),\gamma'(\theta(t)))\le\kappa$ for all $t\in J$ (notice in particular that the orientations of $\gamma$ and $\gamma'$ are required to be the same on their ``overlap'').
The following proposition is a slight elaboration on Theorem \ref{loxo}.

\begin{prop}\textbf{(Axes of random isometries pass close to the basepoint)}\label{hyperbolic}
Let $G$ be a group acting by isometries on a separable geodesic Gromov hyperbolic space $X$, and let $\mu$ be a nonelementary probability measure on $G$ with finite support. For all $K>0$, there exists $\kappa>0$ such that for $\mathbb{P}$-a.e. sample path $(\Phi_n)_{n\in\mathbb{N}}$ of the random walk on $(G,\mu)$, and all $\epsilon\in (0,1)$,  
 there exists $n_0\in\mathbb{N}$ such that for all $n\ge n_0$, the element $\Phi_n$ acts loxodromically on $X$, and for all geodesic segments $\gamma$ from $x_0$ to $\Phi_n.x_0$, every $K$-quasi-axis of $\Phi_n$ crosses a subsegment of $\gamma$ of length at least $(1-\epsilon)d_X(x_0,\Phi_n.x_0)$ up to distance $\kappa$.
\end{prop}

We will now prove Proposition \ref{hyperbolic}.  
 Our proof follows the same argumentation as in recent work by Taylor--Tiozzo \cite{TT}, and is based on the following fact. 

\begin{prop}\label{estimate}
Let $G$ be a countable group acting by isometries on a separable Gromov hyperbolic metric space $(X,d_X)$, let $x_0\in X$, and let $\mu$ be a nonelementary probability measure on $G$ with finite support. Then for $\mathbb{P}$-a.e. sample path $(\Phi_n)_{n\in\mathbb{N}}$ of the random walk on $(G,\mu)$, and all $\epsilon>0$, there exists $n_0\in\mathbb{N}$ such that for all $n\ge n_0$, we have $(\Phi_n.x_0|\Phi_n^{-1}.x_0)_{x_0}\le\epsilon n$.
\end{prop} 

The idea behind the proof of Proposition \ref{estimate} is that the first increments in the product $\Phi_n$ are independent from the first increments in the product $\Phi_n^{-1}$. Hyperbolicity of $X$ then forces the geodesic segments $[x_0,\Phi_n.x_0]$ and $[x_0,\Phi_n^{-1}.x_0]$ to diverge rapidly with high probability. To make this precise, we will use the following two lemmas. These basically follow from Maher--Tiozzo's work.

\begin{lemma}(Maher--Tiozzo \cite[Lemmas 5.9 and 5.11]{MT14})\label{estimate-1}
Let $G$ be a countable group acting by isometries on a separable Gromov hyperbolic space $(X,d_X)$, and let $\mu$ be a nonelementary probability measure on $G$ with finite support. Let $L_X>0$ denote the drift of the random walk on $(G,\mu)$ with respect to $d_X$. Then for all $\epsilon >0$, there exists $C>0$ such that for all $n\in\mathbb{N}$, one has $$\mathbb{P}\left((\Phi_n.x_0|\Phi_{\lfloor \frac{\epsilon n}{2}\rfloor}.x_0)_{x_0}\ge\frac{\epsilon L_X.n}{4}\right)\ge 1-e^{-Cn}$$ and $$\mathbb{P}\left((\Phi_n^{-1}.x_0|\Phi_n^{-1}\Phi_{n-\lfloor \frac{\epsilon n}{2}\rfloor}.x_0)_{x_0}\ge\frac{\epsilon L_X.n}{4}\right)\ge 1-e^{-Cn}$$ and $$\mathbb{P}\left((\Phi_n^{-1}\Phi_{n-\lfloor \frac{\epsilon n}{2}\rfloor}.x_0|\Phi_{\lfloor \frac{\epsilon n}{2}\rfloor}.x_0)_{x_0}\le\frac{\epsilon L_X.n}{5}\right)\ge 1-e^{-Cn}.$$
\end{lemma}

\begin{proof}
This is a version of \cite[Lemmas 5.9 and 5.11]{MT14}, where basically one only replaces $m=\lfloor \frac{n}{2}\rfloor $ by $\lfloor \frac{\epsilon n}{2}\rfloor$ (all arguments work in the same way). As noticed in Maher--Tiozzo's paper, the exponential rate follows from exponential decay of shadows, which was already established in \cite{Mah12}. The first two estimates play symmetric roles and correspond to \cite[Lemma 5.11]{MT14}, and the third one corresponds to \cite[Lemma 5.9]{MT14}.
\end{proof}

\begin{lemma}(Maher--Tiozzo \cite[Lemma 5.9]{MT14})\label{Gromov-product}
Let $X$ be a Gromov hyperbolic space, and let $x_0\in X$. Then there exist $\kappa_1,\kappa_2>0$ that only depend on the hyperbolicity constant of $X$ such that for all $a,b,c,d\in X$, if there exists $A>0$ such that $(a|b)_{x_0}\ge A$, $(c|d)_{x_0}\ge A$ and $(a|c)_{x_0}\le A-\kappa_1$, then $|(a|c)_{x_0}-(b|d)_{x_0}|\le\kappa_2$.
\end{lemma}

\begin{proof}[Proof of Proposition \ref{estimate}]
For all $\epsilon>0$, there exists $C>0$ such that all three events considered in Lemma \ref{estimate-1} happen simultaneously with probability greater than $1-e^{-Cn}$. Applying Lemma \ref{Gromov-product} to $a:=\Phi_{\lfloor\frac{\epsilon n}{2}\rfloor}.x_0$, $b:=\Phi_n.x_0$, $c:=\Phi_n^{-1}\Phi_{n-\lfloor\frac{\epsilon n}{2}\rfloor}.x_0$ and $d:=\Phi_n^{-1}.x_0$, with $A=\frac{\epsilon L_X.n}{4}$, one deduces that for all $\epsilon>0$, there exists $C>0$ such that for all $n\in\mathbb{N}$, one has $$\mathbb{P}\left((\Phi_n.x_0|\Phi_n^{-1}.x_0)_{x_0}\le\epsilon n\right)\ge 1-e^{-Cn}.$$ Proposition \ref{estimate} then follows by applying the Borel--Cantelli Lemma.
\end{proof}

\begin{proof}[Proof of Proposition \ref{hyperbolic}] 
 Consider a  $K$-quasi-axis for $\Phi$, and $y_0$ a closest point to $x_0$ on it. By hyperbolicity and the minimising property for $y_0$, all geodesic segments $[\Phi^{-1}. x_0, \Phi^{-1}. y_0]$, $[\Phi^{-1}. y_0, \Phi. y_0]$ and $[\Phi .y_0 \Phi. x_0]$ remain $\kappa_0=\kappa_0(K,\delta)$-close to any geodesic  $[\Phi^{-1}. x_0,\Phi. x_0]$. Since the lengths of $[\Phi^{-1}. x_0, \Phi^{-1}. y_0]$, $[x_0, y_0]$ and $[\Phi x_0, \Phi y_0]$ are  within $\kappa_1(K,\delta)$  of  $   (\Phi. x_0| \Phi^{-1}.x_0 )_{x_0} $, we know that for every $K>0$, there exists $\kappa=\kappa(K,\delta)$ (where $\delta$ is the hyperbolicity constant of $X$) such that all $K$-quasi-axes of loxodromic isometries $\Phi$ cross the central subsegment of length $d_X(x_0, \Phi. x_0) -2 (\Phi. x_0| \Phi^{-1}.x_0 )_{x_0}$ of any geodesic segment $[x_0, \Phi. x_0]$, up to distance $\kappa$. Proposition \ref{hyperbolic} is therefore a consequence of Theorem \ref{drift} (Positive drift) and Proposition \ref{estimate}, together with the fact that $\mathbb{P}$-a.s., the element $\Phi_n$ is loxodromic for all $n$ large enough (Theorem \ref{loxo}).
\end{proof}

 Given $\epsilon>0$, a quasi-geodesic ray $\gamma:\mathbb{R}_+\to X$, and $n\in\mathbb{N}$, we denote by $t_1^{\epsilon,\gamma}(n)$ (resp. $t_2^{\epsilon,\gamma}(n)$) the infimum of all real numbers such that $d_X(\gamma(0),\gamma(t_1(n)))\ge \epsilon L_X.n$ (resp. $d_X(\gamma(0),\gamma(t_2(n)))\ge (1-\epsilon)L_X.n$). Combining Theorem \ref{hyp-cv} (convergence to the boundary), Theorem \ref{souslin} (sublinear tracking) and Proposition \ref{hyperbolic} (axes of random isometries pass close to the basepoint), one can establish the following. Details of the proof are left to the reader.

\begin{prop}\label{hyp-drift}
Let $X$ be a separable geodesic Gromov hyperbolic metric space, with hyperbolicity constant $\delta$. For all $K>0$, the exists $\kappa=\kappa(K,\delta)$, such that the following holds. Let $G$ be a group acting by isometries on $X$, and let $\mu$ be a nonelementary probability measure on $G$ with finite support. Let $L_X>0$ denote the drift of the random walk on $(G,\mu)$ with respect to $d_X$. Then for $\mathbb{P}$-a.e. sample path $\mathbf{\Phi}:=(\Phi_n)_{n\in\mathbb{N}}$ of the random walk on $(G,\mu)$, all $\epsilon>0$, and all $K$-quasi-geodesic rays $\gamma:\mathbb{R}_+\to X$ converging to $\text{bnd}(\mathbf{\Phi})$, there exists $n_0\in\mathbb{N}$ such that for all $n\ge n_0$, any $K$-quasi-axis of $\Phi_n$ crosses $\gamma_{|[t_1^{\epsilon,\gamma}(n),t_2^{\epsilon,\gamma}(n)]}$ up to distance $\kappa$.       
\qed
\end{prop}

\section{Translation lengths of random elements: abstract setting}\label{abstract}

The goal of this section is to provide the general setting for the \emph{lifting argument} we wish to use for proving our spectral theorems. We will establish sufficient conditions under which, if $G$ is a group acting by isometries on two metric spaces $Y$ and $X$, where $X$ is a hyperbolic coarse Lipschitz $G$-equivariant image of $Y$, then elements obtained along a typical sample path of the random walk on $G$ have linearly growing translation length in $Y$. The main result of this section is Theorem \ref{spectral-gen} below.

Let $G$ be a countable group acting by isometries on a (possibly asymmetric) geodesic metric space $(Y,d_Y)$ (i.e. $d_Y$ need not be symmetric, but it satisfies the triangle inequality, as well as the separation property  that $d_Y(a,b)=0$ if and only if $a=b$). 
 We denote by $d_Y^{sym}$ the symmetrized version of the metric, defined by letting $d_Y^{sym}(x,y):=d_Y(x,y)+d_Y(y,x)$ for all $x,y\in Y$. A \emph{bordification} of $Y$ is a Hausdorff, second-countable topological space $\overline{Y}$ such that $Y$ is homeomorphic to an open dense subset of $\overline{Y}$, and such that the $G$-action on $Y$ extends to an action on $\overline{Y}$ by homeomorphisms. We let $\partial Y:=\overline{Y}\smallsetminus Y$. From now on, when we talk about a \emph{metric $G$-space} $(Y,d_Y)$, we assume that the $G$-action is by isometries, but we do not necessarily assume that the metric $d_Y$ is symmetric.   
 Note that one can talk about geodesics in a non-symetric metric space $(Y, d_Y)$: a geodesic is a map $\gamma: I\to Y$ where $I$ is an interval in $\mathbb{R}$, for which, for all $a<b$ in $I$, $d_Y(\gamma(a), \gamma(b)) = b-a$.  

\begin{de}\textbf{(Boundary convergence)}
 Let $G$ be a countable group, let $Y$ be a metric $G$-space which admits a bordification $\overline{Y}$. Let $\mu$ be a probability measure on $G$. We say that $(G,\mu)$ is \emph{$\overline{Y}$-boundary converging} if for $\overline{\mathbb{P}}$-a.e. bilateral sample path $\mathbf{\Phi}:=(\Phi_n)_{n\in\mathbb{Z}}$ of the random walk on $(G,\mu)$, there exist $\text{bnd}^-(\mathbf{\Phi}),\text{bnd}^+(\mathbf{\Phi})\in\partial Y$ such that for all $y_0\in Y$, the sequence $(\Phi_n.y_0)_{n\in\mathbb{N}}$ converges to $\text{bnd}^+(\mathbf{\Phi})$, and the sequence $(\Phi_{-n}.y_0)_{n\in\mathbb{N}}$ converges to $\text{bnd}^-(\mathbf{\Phi})$. 
\end{de}

\begin{de}\textbf{(Hyperbolic electrification)}
A \emph{hyperbolic $G$-electrification} is a pair of $G$-spaces $(Y,X)$, together with a map $\pi:Y\to X$, such that

\begin{itemize}
\item the space $(Y,d_Y)$ is a (possibly asymmetric) geodesic metric $G$-space, and
\item the space $(X,d_X)$ is a separable Gromov hyperbolic (symmetric) geodesic metric $G$-space, and
\item there exists $K_1>0$ such that for all $x,y\in Y$ and all $g\in G$, we have $$d_X(g\pi(x),\pi(gx))\le K_1$$ (we say that $\pi$ is \emph{coarsely $G$-equivariant}) and $$d_X(\pi(x),\pi(y))\le K_1d_Y(x,y)+K_1$$ (we say that $\pi$ is \emph{coarsely Lipschitz}), and
\item there exists $K_2>0$ such that for all $x,y\in Y$, there exists a geodesic segment from $x$ to $y$ whose $\pi$-image in $X$ is a $K_2$-unparameterized quasigeodesic.
\end{itemize}
\end{de}

We fix from now on a hyperbolic $G$-electrification $(Y,X)$, a map $\pi:Y\to X$, and real numbers $K_1,K_2>0$ given by the above definition. We also assume that $Y$ admits a bordification $\overline{Y}$. From now on, nonelementarity of a probability measure on $G$ will always be meant with respect to the $G$-action on $X$. Given an interval $I\subseteq\mathbb{R}$, we say that a geodesic $\gamma:I\to Y$ is \emph{electric} if $\pi\circ\gamma:I\to X$ is a $K_2$-unparameterized quasigeodesic (the last property in the above definition states the existence of an electric geodesic between any two points in $Y$). We denote by $\Gamma_{el} Y$ the collection of all electric geodesics in $Y$. We let $\kappa=\kappa(K_2,\delta)$ be the constant given by Proposition \ref{hyp-drift}. 

\begin{de}\label{de-contr}\textbf{(Contraction)}
Let $\gamma:\mathbb{R}\to Y$ be an electric geodesic line, and let $y\in Y$ be a point that lies on the image of $\gamma$. Given $B,D>0$, we say that $\gamma$ is \emph{$(B,D)$-contracting at $y$} if the following holds. Let $a\in\mathbb{R}$ be such that $\gamma(a)=y$, and let $b\in\mathbb{R}$ be the infimum of all real numbers such that $\pi\circ\gamma([a,b])$ has $d_X$-diameter at least $B$. Then for all geodesic lines $\gamma':\mathbb{R}\to Y$, if $\pi\circ\gamma'$ crosses $\pi\circ\gamma_{|[a,b]}$ up to distance $\kappa$, then there exists $a'\in\mathbb{R}$ such that $d_Y(\gamma'(a'),\gamma(a))\le D$.
\end{de}

The following definition was first given in \cite[Definition 6.13]{AK}: roughly speaking, a point $y\in Y$ is high if for all $y'\in Y$, it is shorter to go from $y$ to $y'$ than to go from $y'$ to $y$.

\begin{de}\textbf{(Highness)}
Given $M>0$, we say that a point $y\in Y$ is \emph{$M$-high} if for all $y'\in Y$, we have $d_Y^{sym}(y',y)\le M d_Y(y',y)$.
\end{de}

\noindent We define the density $\theta(A)$ of a subset $A\subseteq\mathbb{Z}$ as $$\theta(A):=\min\left\{\liminf_{n\to +\infty}\frac{|A\cap [-n,0]|}{n},\liminf_{n\to +\infty}\frac{|A\cap [0,n]|}{n}\right\}.$$ 
The following definition is inspired by \cite[Theorem 6]{Tio13}.

\begin{de}\label{pencil}\textbf{(Contracting pencils of geodesics)}
Let $G$ be a countable group, let $(Y,X)$ be a hyperbolic $G$-electrification, and let $y_0\in Y$. Let $\mu$ be a nonelementary probability measure on $G$, such that the random walk on $(G,\mu)$ is $\overline{Y}$-boundary converging. Denote by $\check{\nu}$ and $\nu$ the hitting measures on $\partial Y$ for the random walks on $(G,\check{\mu})$ and $(G,\mu)$. We say that $(G,\mu)$ \emph{has contracting pencils of geodesics} if there exists a $G$-equivariant map $P:\partial Y\times\partial Y\to \mathcal{P}(\Gamma_{el} Y)$ which associates to any pair of points $(x,y)\in\partial Y$ a set of electric geodesics in $Y$ joining $x$ to $y$, so that  
\begin{itemize}
\item the map 
\begin{displaymath}
\begin{array}{cccc}
D:&\partial Y\times \partial Y&\to & \mathbb{R}\\
& (x,y)&\mapsto& \sup_{\gamma\in P(x,y)}d_Y^{sym}(y_0,\gamma(\mathbb{R}))
\end{array}
\end{displaymath}
\noindent is measurable and $\check{\nu}\otimes\nu$-a.e. finite, and
\item for all $\theta\in (0,1)$, there exist $B,D,M>0$ such that for $\overline{\mathbb{P}}$-a.e. bilateral sample path $\mathbf{\Phi}:=(\Phi_n)_{n\in\mathbb{Z}}$ of the random walk on $(G,\mu)$, and all geodesic lines $\gamma\in P(\text{bnd}^-(\mathbf{\Phi}),\text{bnd}^+(\mathbf{\Phi}))$, the set of integers $k\in\mathbb{Z}$ such that for all $d_Y^{sym}$-closest-point projections $z$ of $\Phi_k.y_0$ to $\gamma$, the point $z$ is $M$-high, and $\gamma$ is $(B,D)$-contracting at $z$, has density at least $\theta$.
\end{itemize} 
\end{de}

Let $\Phi\in G$ be a loxodromic isometry of $X$. We say that $\Phi$ \emph{has an electric translation axis}  
 in $Y$ if there exists a $\Phi$-invariant electric geodesic line $\gamma:\mathbb{R}\to Y$ such that for all $y\in\gamma(\mathbb{R})$, the line $\gamma$ restricts to a geodesic segment from $\Phi.y$ to $y$, and $$d_Y(\Phi.y,y)=\inf_{y'\in Y}d_Y(\Phi.y',y').$$ We then denote by $l_Y(\Phi)$ the \emph{translation length} of $\Phi$, defined as $l_Y(\Phi):= d_Y(\Phi.y,y)$ for all $y$ lying on the image of an electric axis of $\Phi$. We notice that the translation length of $\Phi$ is well-defined, i.e. it does not depend on the choice of an axis for $\Phi$, and in addition, for all $y\in Y$, we have $l_Y(\Phi)\le d_Y(\Phi.y,y)$.

 Notice that with the above definition, the action of $\Phi$ on an electric translation axis is by translation in the negative direction. Therefore, with our definitions, any electric geodesic axis for $\Phi$ in $Y$ projects (as an oriented line) to a $K_2$-quasi-axis for $\Phi^{-1}$ in $X$. 
These conventions may sound a bit awkward, however they will turn out to be quite natural in the context where $Y$ is either the Teichmüller space of a surface or Culler--Vogtmann's outer space, because the natural actions of either $\text{Mod}(S)$ or $\text{Out}(F_N)$ on these spaces are right actions (so that one needs to pass to inverses when considering left actions).

\begin{theo}\label{spectral-gen}
Let $G$ be a countable group, and let $(Y,X)$ be a $G$-hyperbolic electrification, with a basepoint $y_0\in Y$. Assume that $Y$ admits a bordification $\overline{Y}$. Let $\mu$ be a nonelementary probability measure on $G$ with finite support. Assume that all elements of $G$ acting loxodromically on $X$ have an electric translation axis in $Y$. If the random walk on $(G,\mu)$ is $\overline{Y}$-boundary converging and has contracting pencils of geodesics, then for $\mathbb{P}$-a.e. sample path $(\Phi_n)_{n\in\mathbb{N}}$ of the random walk on $(G,\mu)$, we have $$\lim_{n\to +\infty} \frac{l_Y(\Phi_n^{-1})}{n}=L_Y,$$ where $L_Y$ denotes the drift of the random walk on $(G,\mu)$ with respect to $d_Y$. 
\end{theo}

\begin{rk}\label{rk-abstract}
Since the support of $\mu$ is assumed to be finite, it follows from Theorem \ref{loxo} that for $\mathbb{P}$-a.e. sample path $(\Phi_n)_{n\in\mathbb{N}}$ of the random walk on $(G,\mu)$, eventually all elements $\Phi_n$ act as loxodromic isometries of $X$, which justifies that it makes sense to write $l_Y(\Phi_n^{-1})$ in Theorem \ref{spectral-gen}. Combining Benoist--Quint's arguments from \cite{BQ14} with Maher--Tiozzo's results on random walks on nonproper hyperbolic spaces \cite{MT14}, one can actually show that Proposition \ref{hyperbolic} remains valid if $\mu$ is only assumed to have finite second moment with respect to $d_X$ (see \cite[Section 2]{Hor15}). Therefore, Theorem \ref{spectral-gen} remains true if $\mu$ is only assumed to have finite second moment with respect to $d_Y$.

If the support of $\mu$ is only assumed to have finite first moment with respect to $d_Y$, then we don't know whether infinitely many $\Phi_n$ fail to satisfy the conclusion of Proposition \ref{hyperbolic}  (or even fail to be loxodromic),     
however the conclusion of Theorem \ref{spectral-gen} remains valid if one takes the limit along the subsequence corresponding to those integers $n$ for which $\Phi_n$ satisfies the conclusion of Proposition \ref{hyperbolic}. Therefore, Theorem \ref{spectral-gen} remains valid in this situation if one replaces the limit by a limsup.
\end{rk}

The general idea of the proof of Theorem \ref{spectral-gen} is the following, it is illustrated in Figure \ref{fig-gen}. Typical sample paths of the realization in $Y$ of the random walk on $(G,\mu)$ stay close to a (random) geodesic ray $\tau_Y$. Using the existence of contracting pencils of geodesics, for all $n\in\mathbb{N}$, we can find a contracting subsegment $J$ of the image of $\tau_Y$, whose distance to $y_0$ is small compared to $L_Y.n$ (Step 1 in the proof below). Using the properties of random walks on isometry groups of hyperbolic spaces recalled in Section \ref{sec1}, one can also arrange so that every quasi-axis of $\Phi_n$ in $X$ crosses $\pi(J)$ up to distance $\kappa$ (Step 2). The contraction property implies that $\Phi_n^{-1}$ has an electric axis in $Y$ that passes at small distance from $y_0$ (Step 3). This is enough to deduce that the translation length of $\Phi_n^{-1}$ in $Y$ is close to $L_Y.n$, as required (Step 4).

\begin{figure}
\begin{center}
\includegraphics[scale=.6]{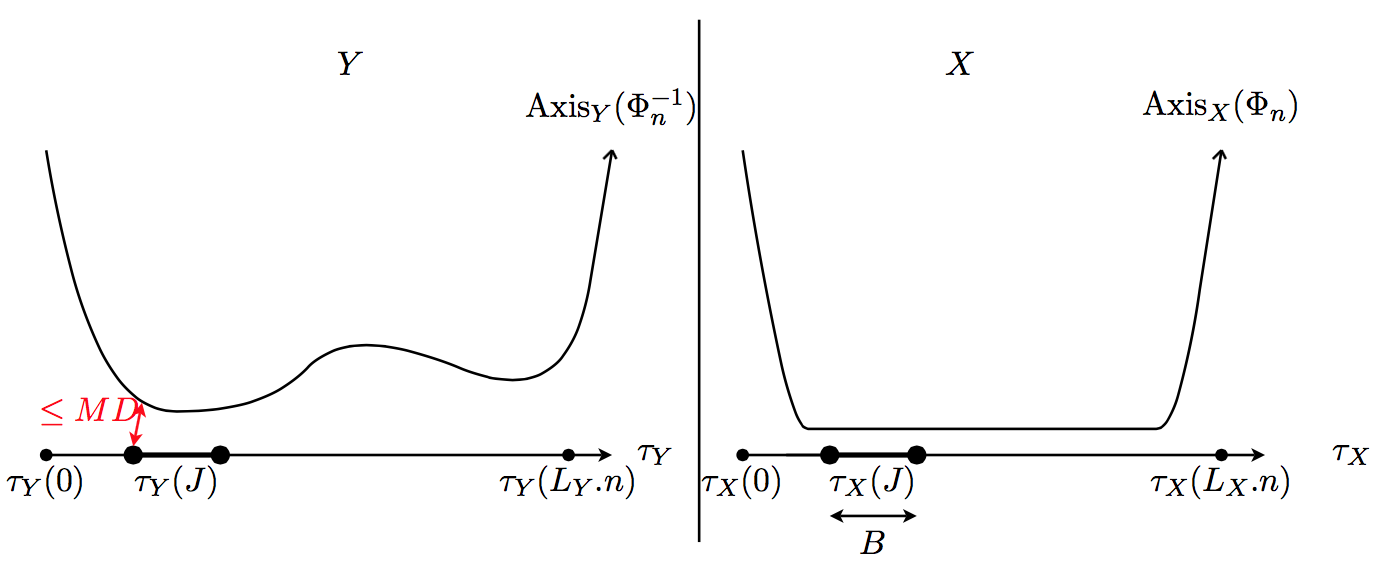}
\caption{The situation in the proof of Theorem \ref{spectral-gen}.}
\label{fig-gen}
\end{center}
\end{figure}

\begin{proof}[Proof of Theorem \ref{spectral-gen}]
By Theorem \ref{drift}, the drift $L_X$ of the realization of the random walk on $X$ (with respect to the metric $d_X$) is positive. Since $\pi$ is coarsely $G$-equivariant and coarsely Lipschitz, this implies that the drift $L_Y$ of the realization of the random walk on $Y$ (with respect to the metric $d_Y$) is also positive. 

Let $K_1\ge 1$ be a constant (provided by the definition of a hyperbolic $G$-electrification) such that for all $x,y\in Y$, we have $$d_X(\pi(x),\pi(y))\le K_1 d_Y(x,y)+K_1.$$ Let $K_2\ge 1$ be a  constant such that for all electric geodesics $\gamma:I\to Y$ (where $I\subseteq\mathbb{R}$ is an interval), we have that $\pi\circ\gamma:I\to X$ is an unparameterized $K_2$-quasigeodesic. 
 
Let $\epsilon>0$. We fix a basepoint $y_0\in Y$, and let $x_0:=\pi(y_0)$. Since $(G,\mu)$ is $\overline{Y}$-boundary converging, we have that for $\mathbb{P}$-a.e. sample path $\mathbf{\Phi}:=(\Phi_n)_{n\in\mathbb{N}}$ of the random walk on $(G,\mu)$, the sequence $(\Phi_n.y_0)_{n\in\mathbb{N}}$ converges to a point $\text{bnd}_Y(\mathbf{\Phi})\in \partial Y$, and by Theorem \ref{hyp-cv}, the sequence $(\Phi_n.x_0)_{n\in\mathbb{N}}$ converges to a point $\text{bnd}_X(\mathbf{\Phi})\in\partial_{\infty}X$.  

Let $\theta:=\frac 7 8$, and let $B,D>0$ and $M\ge 1$ be the corresponding constants provided by the definition of having contracting pencils of geodesics (Definition \ref{pencil}). 
\\
\\
\textbf{Claim.} For $\mathbb{P}$-a.e. sample path $\mathbf{\Phi}:=(\Phi_n)_{n\in\mathbb{N}}$ of the random walk on $(G,\mu)$, there exists an integer $n_0\in\mathbb{N}$ and an electric geodesic ray $\tau_Y:\mathbb{R}_+\to Y$, such that for all $n\ge n_0$, the element $\Phi_n$ acts as a loxodromic isometry of $X$, and letting $\tau_X:=\pi\circ\tau_Y$, we have 
\begin{itemize}
\item[$(i)$] $|d_{Y}(y_0,\Phi_n.y_0)-L_{Y}.n|\le\epsilon L_{Y}. n$, and
\item[$(ii)$] $d_{Y}^{sym}(\Phi_n.y_0,\tau_{Y}(L_{Y}.n))\le\frac{\min\{L_{Y},L_{X}\}}{6K_1}.n-1$, and 
\item[$(iii)$] $|d_{X}(\tau_X(0),\Phi_n.x_0)-L_X.n|\le\frac{L_{X}.n}{12}$, and
\item[$(iv)$] there are at least $\frac{6n}{7}$ integers $k\in\{0,\dots,n\}$ such that
   for all $d_Y^{sym}$-closest-point projections
   $z$ of $\Phi_k.y_0$ to $\tau_Y$, 
   the point $z$ is $M$-high, and $\tau_Y$ is $(B,D)$-contracting at $z$, and 
\item[$(v)$] every $K_2$-quasi-axis of $\Phi_n$ in $X$ crosses $\tau_{X}([t_1(n),t_2(n)])$ up to distance $\kappa$, where $\kappa=\kappa(K_2,\delta)$ is the constant provided by Proposition \ref{hyp-drift}, and $t_1(n)$ (resp. $t_2(n)$) is the infimum of all real numbers such that $d_X(\tau_X(0),\tau_X(t_1(n)))\ge \epsilon L_X.n$ (resp. $d_X(\tau_X(0),\tau_X(t_2(n)))\ge (1-\epsilon)L_X.n$).
\end{itemize}

\begin{proof}[Proof of the Claim.] Let us explain how to choose $\tau_Y$. For $\overline{\mathbb{P}}$-a.e. bilateral path $\mathbf{\Phi}$ of the random walk on $(G,\mu)$, Tiozzo produces in \cite[Theorem 6]{Tio13} a bi-infinite geodesic between the boundary limit points $\text{bnd}^-(\mathbf{\Phi})$ and $\text{bnd}^+(\mathbf{\Phi})$, in the pencil of geodesics $D(\text{bnd}^-(\mathbf{\Phi}),\text{bnd}^+(\mathbf{\Phi}))$, from which the random walk has sublinear tracking. We choose $\tau_Y$ to be a positive ray in this geodesic. Property $(ii)$ follows from the sublinear tracking, and property $(iv)$ follows from the properties of our pencil of geodesics.

Properties $(i)$ and $(iii)$ come from the definition of the drifts of the random walk on $Y$ and $X$. 
 Property $(v)$   
 is   guaranteed by Proposition \ref{hyp-drift}.
\end{proof}

We let $\mathbf{\Phi}:=(\Phi_n)_{n\in\mathbb{N}}$ be a sample path that satisfies the conclusion of the above claim. By choosing $n_0$ sufficiently large, we may also assume that $B+K_1\le L_X.n_0$, that $D\le L_Y.n_0$, that $d_Y^{sym}(y_0,\tau_Y(0))\le L_Y.n_0$ and that $2K_1\le L_X.n_0$. 
\\
\\
\textbf{Step 1: We show the existence of a point $y$ close to $y_0$ at which $\tau_Y$ is contracting.}
\\
\\
Let $n\in\mathbb{N}$ be such that $\epsilon n\ge n_0$. Then the interval $[5\epsilon n,6\epsilon n]$ contains an integer $n_2$ 
 {fulfilling the} 
Property $(iv)$ above. In addition, by Property $(ii)$, we have $$d_{Y}^{sym}(\Phi_{n_2}.y_0,\tau_{Y}(L_{Y}.n_2))\le\frac{\min\{L_{Y},L_{X}\}.n_2}{6K_1}-1,$$ so in particular we have $$d_{Y}^{sym}(\Phi_{n_2}.y_0,\tau_{Y}(L_{Y}.n_2))\le\frac{L_{Y}.n_2}{6}\le \epsilon L_{Y}.n.$$ Let $y$ be a $d_{Y}^{sym}$-closest point projection of $\Phi_{n_2}.y_0$ to $\tau_Y$. By definition of a closest-point projection, we necessarily have $$d_{Y}^{sym}(\Phi_{n_2}.y_0,y)\le\epsilon L_{Y}.n.$$ It then follows from the triangle inequality that $$d_{Y}^{sym}(y,\tau_{Y}(L_{Y}.n_2))\le 2\epsilon L_{Y}.n,$$ so $y\in\tau_{Y}([3\epsilon L_{Y}.n,8\epsilon L_{Y}.n])$. 
\\
\\
We now let $a\in\mathbb{R}$ be such that $\tau_Y(a)=y$, and let $b\ge a$ be the infimum of all real numbers such that $\tau_X([a,b])$ has diameter at least $B$. We let $J:=[a,b]$.   
By definition for all $\epsilon >0$,  $\tau_X([a,b-\epsilon])$ has diameter at most $B$,   thus  $\tau_X(J)$ has diameter at most $B+K_1$.
\\
\\
\textbf{Step 2: We show that every $K_2$-quasi-axis of $\Phi_n$ in $X$ crosses $(\tau_X)_{|J}$ up to distance $\kappa$.}
\\
\\
Recall from above that $$d_Y^{sym}(\Phi_{n_2}.y_0,\tau_Y(L_Y.n_2))\le\frac{L_X.n_2}{6K_1}-1,$$ and hence by definition of a closest-point projection $$d_Y^{sym}(\Phi_{n_2}.y_0,y)\le\frac{L_X.n_2}{6K_1}-1.$$ Using the fact that $\pi$ is coarsely Lipschitz and that $d_Y\le d_Y^{sym}$, we get that $$d_{X}(\pi(\Phi_{n_2}.y_0),\pi(y))\le\epsilon L_{X}.n.$$ By Property $(iii)$ above, we also have $$|d_{X}(\Phi_{n_2}.x_0,\tau_{X}(0))-L_X.n_2|\le\frac{L_{X}.n_2}{12}\le\frac{\epsilon L_{X}.n}{2}.$$ Finally, since $\pi$ is coarsely $G$-equivariant, we have $$d_X(\pi(\Phi_{n_2}.y_0),\Phi_{n_2}.x_0)\le K_1\le\frac{\epsilon L_X.n}{2}.$$ Together with the triangle inequality, the above three inequalities imply that $$|d_{X}(\pi(y),\tau_{X}(0))-L_X.n_2|\le 2\epsilon L_{X}.n.$$ As $B+K_1\le\epsilon L_X.n$, and $\tau_X(J)$ has diameter at most $B+K_1$, this implies that $$|d_{X}(\pi(y'),\tau_{X}(0))-L_X.n_2|\le 3\epsilon L_{X}.n$$ for all $y'\in \tau_Y(J)$. 
By Property $(v)$ above, we know in addition
that every $K_2$-quasi-axis of $\Phi_n$ in $X$ crosses $\tau_{X}([t_1(n),t_2(n)])$ up to distance $\kappa$ (where $t_1(n)$ and $t_2(n)$ are defined as in Property $(v)$). Assuming that $\epsilon>0$ had been chosen small enough such that $9\epsilon\le 1-\epsilon$, we get that every $K_2$-quasi-axis of $\Phi_n$ in $X$ crosses $(\tau_X)_{|J}$ up to distance $\kappa$.
\\
\\
\textbf{Step 3: We show that the element $\Phi_n^{-1}$ has an electric axis in $Y$ that passes close to $y_0$.}
\\
\\ 
By hypothesis, all elements of $G$ acting loxodromically on $X$ have an electric axis. In particular, the element $\Phi_n^{-1}$ has an axis in $Y$, whose projection to $X$ is a $K_2$-quasi-axis for $\Phi_n$, and hence crosses $(\tau_X)_{|J}$ up to distance $\kappa$. Since $\tau_Y$ is $(B,D)$-contracting at $y$ (Property $(iv)$), there exists a point $y'$ lying on the axis of $\Phi_n^{-1}$ in $Y$ such that $d_{Y}(y',y)\le D$, and hence $d_Y^{sym}(y,y')\le MD$ since $y$ is $M$-high. Hence we have
\begin{displaymath}
\begin{array}{rl}
d_{Y}^{sym}(y_0,y')&\le d_{Y}^{sym}(y_0,\tau_Y(0))+d_{Y}^{sym}(\tau_Y(0),y)+d_{Y}^{sym}(y,y')\\
&\le d_{Y}^{sym}(y_0,\tau_Y(0))+Md_{Y}(\tau_Y(0),y)+d_{Y}^{sym}(y,y')\\
&\le d_{Y}^{sym}(y_0,\tau_Y(0))+8M\epsilon L_{Y}.n+MD\\
&\le 10M\epsilon L_{Y}.n.
\end{array}
\end{displaymath}
\noindent Notice here that the second inequality uses the fact that $y$ is $M$-high. 
\\
\\
\textbf{Step 4: End of the proof.}
\\
\\
By $G$-invariance of $d_{Y}^{sym}$, we also have $d_{Y}^{sym}(\Phi_n.y_0,\Phi_n.y')\le 10M\epsilon L_{Y}.n$. Using the triangle inequality, we get $d_{Y}(y',\Phi_n.y')\ge d_{Y}(y_0,\Phi_n.y_0)-20M\epsilon L_{Y}.n$. Property $(i)$ above then implies that $d_{Y}(y',\Phi_n.y')\ge (1-21M\epsilon)L_{Y}.n$. Since $y'$ belongs to the translation axis of $\Phi_n^{-1}$, we obtain $\frac{1}{n} l_Y(\Phi_n^{-1})\ge (1-21M\epsilon)L_{Y}$. As $\epsilon>0$ was chosen arbitrary small, and since the above inequality holds for all sufficiently large $n\in\mathbb{N}$, we have $$\liminf_{n\to +\infty}\frac{1}{n}l_Y(\Phi_n^{-1})\ge L_Y.$$ Notice in addition that for all $n\in\mathbb{N}$, we have $$\frac{1}{n}l_Y(\Phi_n^{-1})\le\frac{1}{n}d_Y(y_0,\Phi_n.y_0),$$ and hence $$\limsup_{n\to +\infty}\frac{1}{n}l_Y(\Phi_n^{-1})\le L_Y.$$ This completes the proof of Theorem \ref{spectral-gen}.
\end{proof}

\section{Spectral theorem for random walks on $\text{Mod}(S)$}\label{sec3}

Let $S$ be a closed, connected, oriented, hyperbolic surface. We will now prove our spectral theorem for random walks on $\text{Mod}(S)$. We fix a hyperbolic metric $\rho$ on $S$. We refer to Section \ref{sec-mod-background} below for definitions.

\begin{theo}\label{spectral}
Let $S$ be a closed, connected, oriented, hyperbolic surface, and let $\mu$ be a nonelementary probability measure on $\text{Mod}(S)$ with finite support. Then for $\mathbb{P}$-a.e. sample path $(\Phi_n)_{n\in\mathbb{N}}$ of the random walk on $(\text{Mod}(S),\mu)$, we have $$\lim_{n\to +\infty}\frac{1}{n}\log(\lambda(\Phi_n))=L_{\mathcal{T}},$$ where $\lambda(\Phi_n)$ denotes the stretching factor of the pseudo-Anosov mapping class $\Phi_n$, and $L_{\mathcal{T}}>0$ is the drift of the random walk on $(\text{Mod}(S),\mu)$ with respect to the Teichmüller metric. 
\end{theo}

\begin{rk}\label{rk-mod}
 We recall that when $\mu$ has finite support, then with probability equal to $1$, all $\Phi_n$ are eventually pseudo-Anosov, so that the stretching factor $\lambda(\Phi_n)$ is well-defined (notice also that in this context, we have $\lambda(\Phi_n)=\lambda(\Phi_n^{-1})$,   
  {as their logarithms are the translation lenghts in Teichm\"uller space equipped with the Teichm\"uller metric which is symmetric}). Theorem \ref{spectral} will be established by applying Theorem \ref{spectral-gen} to the $\text{Mod}(S)$-hyperbolic electrification made of the Teichmüller space $\mathcal{T}(S)$ and the curve graph $\mathcal{C}(S)$. In view of Remark \ref{rk-abstract}, Theorem \ref{spectral} remains valid if $\mu$ is only assumed to have finite second moment with respect to the Teichmüller metric $d_{\mathcal{T}}$. However, if $\mu$ is only assumed to have finite first moment with respect to $d_{\mathcal{T}}$, then one has to replace the limit by a limsup in the statement.
\end{rk}

The section is organized as follows. In Section \ref{sec-mod-background}, we record the necessary background on mapping class groups. In Section \ref{sec-contr-mod}, we establish a sufficient criterion for a subsegment of a Teichmüller geodesic in $\mathcal{T}(S)$ to have the required contraction property: namely, we prove that if the projection of this subsegment to the curve graph $\mathcal{C}(S)$ makes progress, then this subsegment is contracting. We prove in Section \ref{sec3-3} that typical Teichmüller geodesics contain infinitely many such subsegments. This will enable us to conclude the proof of Theorem \ref{spectral} in Section \ref{sec3-4} by applying Theorem \ref{spectral-gen}. 

\subsection{Background on mapping class groups}\label{sec-mod-background}

We review 
 some rather classical background on mapping class groups of surfaces. The familiar reader may skip directly to Section \ref{sec-contr-mod}.

Let $S$ be a closed, connected, oriented, hyperbolic surface. The \emph{mapping class group} $\text{Mod}(S)$ is the group of isotopy classes of orientation-preserving diffeomorphisms of $S$. The group $\text{Mod}(S)$ acts on the Teichmüller space $\mathcal{T}(S)$, and on the curve graph $\mathcal{C}(S)$ of the surface, whose definitions we now recall.

The \emph{Teichmüller space} $\mathcal{T}(S)$ is the space of isotopy classes of conformal  
 structures on $S$. It can be equipped with several metrics; we review the definition of the \emph{Teichmüller metric}. Given $x,y\in\mathcal{T}(S)$, the Teichmüller distance $d_{\mathcal{T}}(x,y)$ is defined as $$d_{\mathcal{T}}(x,y):=\frac{1}{2}\text{inf}_f \log K(f),$$ where the infimum is taken over all quasiconformal diffeomorphisms $f:(S,x)\to (S,y)$ that are isotopic to the identity, and where $K(f)$ denotes the quasiconformal dilatation of $f$. (We note that the Teichmüller metric is a true metric which satisfies the symmetry property.) Teichmüller space, equipped with the Teichmüller metric, is a uniquely geodesic space, and $\text{Mod}(S)$ acts by isometries on $(\mathcal{T}(S),d_{\mathcal{T}})$. 

A simple closed curve $c$ on $S$ is \emph{essential} if it does not bound a disk on $S$. We denote by $\mathcal{S}$ the collection of all isotopy classes of essential simple closed curves on $S$. Given $c\in\mathcal{S}$ and $x\in\mathcal{T}(S)$, we denote by $l_x(c)$ the smallest length of a curve in the isotopy class $c$, measured in the unique hyperbolic metric in the conformal class $x$.  
 Thurston has defined a compactification of $\mathcal{T}(S)$ by taking the closure of the map 
\begin{displaymath}
\begin{array}{cccc}
\mathcal{T}(S)&\to &\mathbb{PR}^{\mathcal{S}}\\
x&\mapsto & (l_x(c))_{c\in\mathcal{S}}
\end{array}
\end{displaymath}
\noindent and he identified the boundary with the space $\mathcal{PMF}$ of projective measured foliations on $S$ (up to Whitehead equivalence), see \cite[Exposé 8]{FLP79}. 

The \emph{curve graph} $\mathcal{C}(S)$ is the graph whose vertices are the isotopy classes of essential simple closed curves on $S$, two vertices being joined by an edge if there are disjoint representatives in the isotopy classes of the corresponding curves. We denote by $d_{\mathcal{C}}$ the simplicial metric on $\mathcal{C}(S)$. Masur and Minsky have proved in \cite{MM99} that $(\mathcal{C}(S),d_{\mathcal{C}})$ is Gromov hyperbolic.

There is a coarsely $\text{Mod}(S)$-equivariant, coarsely Lipschitz map $\pi:\mathcal{T}(S)\to\mathcal{C}(S)$, mapping every point in $\mathcal{T}(S)$ to one of the essential simple closed curves whose length is minimal in $S$. It also follows from the work of Masur and Minsky \cite{MM99} that projections of Teichmüller geodesics to the curve graph of $S$ are uniform unparameterized quasi-geodesics. We sum up the above discussion in the following proposition, with the terminology from Section \ref{abstract}.

\begin{prop}\label{elec-mod}
Let $S$ be a closed, connected, oriented, hyperbolic surface. Then $(\mathcal{T}(S),\mathcal{C}(S))$ is a hyperbolic $\text{Mod}(S)$-electrification. Teichmüller geodesics are electric.
\end{prop}

A measured foliation in $\mathcal{PMF}$ is \emph{arational} if it does not contain any simple closed curve on $S$. It is \emph{uniquely ergodic} if in addition, every foliation in $\mathcal{PMF}$ with the same topological support as $F$ is equal to $F$. We denote by $\mathcal{UE}$ the subspace of $\mathcal{PMF}$ consisting of uniquely ergodic arational foliations, and by $\Delta$ the diagonal subset of $\mathcal{UE}\times\mathcal{UE}$. Given $F,F'\in\mathcal{UE}\times\mathcal{UE}\smallsetminus\Delta$, there is a parameterized \emph{Teichmüller line} $\gamma:\mathbb{R}\to\mathcal{T}(S)$, unique up to precomposition by a translation of $\mathbb{R}$, which is a geodesic line for the Teichmüller metric, and such that $\lim_{t\to -\infty}\gamma(t)=F$ and $\lim_{t\to +\infty}\gamma(t)=F'$ (\cite{GM91}, see also the exposition in \cite[\S 2.5]{Ser12}, for instance). We will usually denote its image by $[F,F']$, which one may still name Teichmüller line, though unparameterized. Similarly, a point $x_0\in\mathcal{T}(S)$ together with a foliation $F\in\mathcal{UE}$ determine a unique \emph{Teichmüller ray}. In addition, given a pair $(F,F')\in\mathcal{UE}\times\mathcal{UE}\smallsetminus\Delta$, and a sequence $((F_n,F'_n))_{n\in\mathbb{N}}\in(\mathcal{UE}\times\mathcal{UE}\smallsetminus\Delta)^{\mathbb{N}}$ that converges to $(F,F')$, the Teichmüller lines $[F_n,F_n']$ admit parameterizations $\gamma_n:\mathbb{R}\to\mathcal{T}(S)$ (as geodesics) that  converge uniformly on compact sets to a geodesic parameterization $\gamma:\mathbb{R}\to\mathcal{T}(S)$  of $[F,F']$.
\\
\\
\indent An element $\Phi\in\text{Mod}(S)$ is \emph{pseudo-Anosov} if there exists a pair of transverse measured foliations $(\mathcal{F}^s,\mu^s)$ and $(\mathcal{F}^u,\mu^u)$ on $S$, together with a real number $\lambda(\Phi)>1$, such that $$\Phi.(\mathcal{F}^s,\mu^s)=(\mathcal{F}^s,\frac{1}{\lambda(\Phi)}\mu^s)$$ and $$\Phi.(\mathcal{F}^u,\mu^u)=(\mathcal{F}^u,\lambda(\Phi)\mu^u).$$ The coefficient $\lambda(\Phi)$ is called the \emph{stretching factor} of $\Phi$. It was shown by Thurston \cite{Thu88} that $\Phi$ is pseudo-Anosov if and only if no power $\Phi^k$ with $k\neq 0, (k\in \mathbb{N})$  fixes the isotopy class of an essential simple closed curve on $S$. The stretching factor $\lambda(\Phi)$ is also equal to the exponential growth rate of the length of any essential simple closed curve on $S$ under iteration of $\Phi$. 

Pseudo-Anosov mapping classes are precisely the elements of $\text{Mod}(S)$ which act as loxodromic isometries of $\mathcal{C}(S)$, see \cite{MM99}. When acting on $\mathcal{T}(S)$, every pseudo-Anosov element $\Phi\in\text{Mod}(S)$ preserves the Teichmüller line $\gamma$ from $\mathcal{F}^s$ to $\mathcal{F}^u$, and acts on $\gamma$ by translation, with $d_{\mathcal{T}}$-translation length equal to $\log\lambda(\Phi)$. With our terminology from the previous section, pseudo-Anosov mapping classes have electric translation axes.
\\
\\
\noindent We finally review work by Kaimanovich--Masur \cite{KM96} concerning random walks on $\text{Mod}(S)$. We note that a subgroup $H$ of $\text{Mod}(S)$ is nonelementary with respect to its action on $\mathcal{C}(S)$ if and only if it contains two pseudo-Anosov elements that generate a nonabelian free subgroup of $H$. Equivalently \cite{McP89}, the subgroup $H$ is nonelementary if $H$ is not virtually cyclic, and $H$ does not virtually preserve the conjugacy class of any essential simple closed curve on $S$. As in the previous sections, we say that a probability measure on $\text{Mod}(S)$ is \emph{nonelementary} if the semigroup generated by its support is a nonelementary subgroup of $\text{Mod}(S)$. Kaimanovich and Masur have proved the following theorem.

\begin{theo}(Kaimanovich--Masur \cite[Theorem 2.2.4]{KM96})\label{KM}
Let $\mu$ be a nonelementary probability measure on $\text{Mod}(S)$. Then for $\mathbb{P}$-a.e. sample path $\mathbf{\Phi}:=(\Phi_n)_{n\in\mathbb{N}}$ of the random walk on $(\text{Mod}(S),\mu)$, there exists a point $\text{bnd}(\mathbf{\Phi})\in\mathcal{UE}$ such that for all $y_0\in\mathcal{T}(S)$, the sequence $(\Phi_n.y_0)_{n\in\mathbb{N}}$ converges to $\text{bnd}(\mathbf{\Phi})$. The hitting measure $\nu$ defined by setting $$\nu(S):=\mathbb{P}(\text{bnd}(\mathbf{\Phi})\in S)$$ for all Borel subsets $S\subseteq \mathcal{PMF}$ is nonatomic. 
\end{theo}

With the terminology from Section \ref{abstract}, the random walk on $\text{Mod}(S)$ is $\overline{\mathcal{T}(S)}$-boundary converging. Kaimanovich--Masur's theorem implies that for $\overline{\mathbb{P}}$-almost every bilateral sample path $\mathbf{\Phi}:=(\Phi_n)_{n\in\mathbb{Z}}$ of the random walk on $(\text{Mod}(S),\mu)$, denoting by $\text{bnd}_-(\mathbf{\Phi})$ (resp. $\text{bnd}_+(\mathbf{\Phi})$) the limit of the sequence $(\Phi_{-n}.y_0)_{n\ge 0}$ (resp. $(\Phi_n.y_0)_{n\ge 0}$), then we have $(\text{bnd}_-(\mathbf{\Phi}),\text{bnd}_+(\mathbf{\Phi}))\in\mathcal{UE}\times\mathcal{UE}\smallsetminus\Delta$. In particular, the Teichmüller geodesic line $[\text{bnd}_-(\mathbf{\Phi}),\text{bnd}_+(\mathbf{\Phi})]$ is $\overline{\mathbb{P}}$-a.s. well-defined.

\subsection{Teichmüller geodesics that make progress are contracting.}\label{sec-contr-mod}

The goal of the next three sections is to prove Theorem \ref{spectral}, by applying Theorem \ref{spectral-gen} to the hyperbolic $\text{Mod}(S)$-electrification $(\mathcal{T}(S),\mathcal{C}(S))$. The key point is to prove that $(\text{Mod}(S),\mu)$ has contracting pencils of geodesics. Proposition \ref{contraction} below will be crucial for establishing the contraction property.

Let $\delta$ be the hyperbolicity constant of $\mathcal{C}(S)$, let $K_2>0$ be a constant such that $\pi$-images of Teichmüller geodesics are $K_2$-unparameterized quasigeodesics in $\mathcal{C}(S)$, and let $\kappa=\kappa(K_2,\delta)$ be the constant provided by Proposition \ref{hyp-drift}. The notion of $(B,D)$-contracting geodesics will be understood with respect to this value of $\kappa$. 

\begin{de}\textbf{(Progress for Teichmüller geodesics)}\\
Given $B_0,C>0$, a Teichmüller line $\gamma:\mathbb{R}\to\mathcal{T}(S)$, and a point $y\in\mathcal{T}(S)$ lying on the image of $\gamma$, we say that $\gamma$ is \emph{$(C;B_0)$-progressing at $y$} if the $\pi$-image in $\mathcal{C}(S)$ of the subsegment of $\gamma$ of $d_{\mathcal{T}}$-length $C$ starting at $y$ has diameter at least $B_0$.
\end{de}

\begin{prop}\label{contraction}
There exists a constant $B_0>0$ such that for all $B\ge B_0$ and all $C>0$, there exists $D>0$, such that for all Teichmüller lines $\gamma:\mathbb{R}\to\mathcal{T}(S)$ and all $y\in\mathcal{T}(S)$ lying on the image of $\gamma$, if $\gamma$ is $(C;B)$-progressing at $y$, then $\gamma$ is $(B,D)$-contracting at $y$. 
\end{prop}

The proof of Proposition \ref{contraction} relies on the following result, due to Dowdall--Duchin--Masur \cite[Theorem A]{DDM14}, which is stated there in terms of the $\epsilon$-thick proportion of $I$. However, the important fact in the proof of \cite[Theorem A]{DDM14} is that $I$ makes definite progress when projected to the curve graph $\mathcal{C}(S)$. 

\begin{prop}(Dowdall--Duchin--Masur \cite[Theorem A]{DDM14})\label{DDM}
For all $\kappa>0$, there exist $B_1,B_2>0$ such that the following holds. For all $C>0$, there exists $D'=D'(C,\kappa)>0$ such that if $[x,y]$ is a Teichmüller segment that contains a subsegment $I$ of $d_{\mathcal{T}}$-length $C$ such that $\text{diam}_{\mathcal{C}(S)}(\pi(I))\ge B_1$, then for all $z\in\mathcal{T}(S)$, if $\pi([x,z])$ crosses $\pi(I)$ up to distance $\kappa$, then $[x,z]$ contains a subsegment $J$ at $d_{\mathcal{T}}$-Hausdorff distance at most $D'$ from a subsegment $I'\subseteq I$, such that $\pi(J)$ has diameter at least $\text{diam}_{\mathcal{C}(S)}(\pi(I))-B_2$.
\end{prop}

\begin{proof}[Sketch of proof]
The proof follows the arguments from the proof of \cite[Theorem A]{DDM14}. We only sketch the argument, and refer the reader to \cite{DDM14} for further details. 

Let $\tau\ge\kappa$ be a constant so that every $K_2$-quasi-triangle in $\mathcal{C}(S)$ is $\tau$-thin. By choosing $B_1>0$ large enough, we can ensure that there exists $w\in I$ such that $d_{\mathcal{C}}(\pi(w),\pi(x))\ge 2\tau+6$, and $d_{\mathcal{C}}(\pi(w),\pi([y,z]))\ge 2\tau +6$. For each such $w$, there exists $u\in [x,z]$ so that $d_{\mathcal{C}}(\pi(w),\pi(u))\le\tau$. Dowdall--Duchin--Masur then prove that $d_V(w,u)$ is bounded (with a bound only depending on $C$) for all proper subsurfaces $V\varsubsetneq S$ (where $d_V$ denotes the distance between the projections in the curve complex of $V$). Using the  distance formula \cite[Proposition 2.4]{DDM14} relating the distance in the Teichmüller space to the sum of the distances between the projections to the curve complexes of the various subsurfaces of $S$, this implies that $d_{\mathcal{T}}(u,w)$ is bounded. The claim follows.
\end{proof}

\begin{proof}[Proof of Proposition \ref{contraction}]
We let $\kappa=\kappa(K_2,\delta)$ be the constant provided by Proposition \ref{hyp-drift}. Let $B_0,C>0$, and let $\gamma:[a,b]\to\mathcal{T}(S)$ be a Teichmüller segment of length $C$ such that $\pi\circ\gamma([a,b])$ has diameter at least $B_0$ in $\mathcal{C}(S)$.  Let $\gamma':\mathbb{R}\to\mathcal{T}(S)$ be a Teichmüller line whose projection to $\mathcal{C}(S)$ crosses $\pi\circ\gamma$ up to distance $\kappa$. This implies that for all $t\in [a,b]$, there exists $t'\in\mathbb{R}$ such that $d_{\mathcal{C}}(\pi(\gamma'(t')),\pi(\gamma(t)))\le\kappa$.

There exists $\kappa'=\kappa'(K_2,\delta)$ such that for $t_2\in\mathbb{R}$ sufficiently large, the $\pi$-image of the Teichmüller segment $[\gamma(a),\gamma'(t_2)]$ crosses $\pi\circ\gamma$ up to distance $\kappa'$. Let $B_1=B_1(\kappa+\kappa')\ge B_1(\kappa')$ and $B_2=B_2(\kappa')$ be the constants provided by Proposition \ref{DDM}. Assume that $B_0\ge B_1+B_2$. Proposition \ref{DDM} implies that $[\gamma(a),\gamma'(t_2)]$ contains a subsegment $J$ at Hausdorff distance at most $D'=D'(C,\kappa')$ from a subsegment of $\gamma([a,b])$, whose projection $\pi(J)$ has diameter at least $B_1$ in $\mathcal{C}(S)$. In particular, the segment $J$ has diameter at most $2D'+C$ in $\mathcal{T}(S)$. In addition, there exists $t_1\in\mathbb{R}$ (chosen sufficiently close to $-\infty$) such that the $\pi$-image of $[\gamma'(t_1),\gamma'(t_2)]$ crosses the $\pi$-image of $J$ up to distance $\kappa+\kappa'$. Let $D'':=D'(2D'+C,\kappa+\kappa')$. Applying Proposition \ref{DDM} again to the segment $[\gamma'(t_1),\gamma'(t_2)]$, we get that $d_{\mathcal{T}}(\gamma',J)\le D''$, so $d_{\mathcal{T}}(\gamma',\gamma([a,b]))\le D'+D''$ and hence $d_{\mathcal{T}}(\gamma(a),\gamma'(\mathbb{R}))\le C+D'+D''$. 
\end{proof}

\subsection{Typical geodesics are infinitely often progressing.}\label{sec3-3}

We now fix once and for all the constant $B_0$ provided by Proposition \ref{contraction}. We also fix a basepoint $y_0\in\mathcal{T}(S)$.

\begin{prop}\label{density-2}
Let $\mu$ be a nonelementary probability measure on $\text{Mod}(S)$. Then for all $\theta\in (0,1)$, there exists $C>0$ such that for $\overline{\mathbb{P}}$-a.e. bilateral sample path $\mathbf{\Phi}:=(\Phi_n)_{n\in\mathbb{N}}$ of the random walk on $(\text{Mod}(S),\mu)$, the set of integers $k\in\mathbb{Z}$ such that all subsegments of $[\text{bnd}_-(\mathbf{\Phi}),\text{bnd}_+(\mathbf{\Phi})]$ of length $C$ in $\mathcal{T}(S)$ starting at a closest-point projection of $\Phi_k.y_0$ have a $\pi$-image of diameter at least $B_0$ in $\mathcal{C}(S)$, has density at least $\theta$.  
\end{prop}

Before proving Proposition \ref{density-2}, we first establish a general lemma.

\begin{lemma}\label{measurable}
Let $X$ be a  topological space, and let $f:X\to\mathbb{R}$ be a map. If $f$ is locally bounded from above, then $f$ is bounded from above by an upper semicontinuous map. 
\end{lemma}

\begin{proof}
Let $f:X\to\mathbb{R}$, and assume that $f$ is locally bounded from above. Let $\widehat{f}:X\to\mathbb{R}$ be the map defined by $$\widehat{f}(x):=\limsup_{x' \to x} f(x')<+\infty.$$ 
 Then $\widehat{f}$ satisfies the required conditions. 
\end{proof}

\begin{lemma}\label{C-mod}
There exists an upper semicontinuous map $C:\mathcal{UE}\times\mathcal{UE}\smallsetminus\Delta\to\mathbb{R}$ such that for $\check{\nu}\otimes\nu$-a.e. $(F^-,F^+)\in\mathcal{UE}\times\mathcal{UE}\smallsetminus\Delta$, all segments of $d_{\mathcal{T}}$-length $C(F^-,F^+)$ starting at a closest-point projection of $y_0$ to the Teichmüller line $[F^-,F^+]$ have a $\pi$-image of diameter at least $B_0$ in $\mathcal{C}(S)$. 
\end{lemma}

\begin{proof}
Let $C'$ be the map defined on $\mathcal{UE}\times\mathcal{UE}\smallsetminus\Delta$, by sending any pair $(F^-,F^+)$ to the infimum of all real numbers $C>0$ such that all segments of $d_{\mathcal{T}}$-length $C$ starting at a closest-point projection of $y_0$ to the Teichmüller line joining $F^-$ to $F^+$, project under $\pi$ to a region of $\mathcal{C}(S)$ of diameter at least $B_0$. In view of Lemma \ref{measurable}, it is enough to check that $C'$ is finite and locally bounded from above.

Assume towards a contradiction that it is not. Then we can find $(F^-,F^+)\in\mathcal{UE}\times\mathcal{UE}\smallsetminus\Delta$, and a sequence $((F_n^-,F_n^+))_{n\in\mathbb{N}}\in (\mathcal{UE}\times\mathcal{UE}\smallsetminus\Delta)^{\mathbb{N}}$ converging to $(F^-,F^+)$, such that for all $n\in\mathbb{N}$, we have $C'((F_n^-,F_n^+))>n$. The Teichmüller geodesics $[F_n^-,F_n^+]$ converge uniformly on compact sets to $[F^-,F^+]$ (for good choices of parameterizations). For all $n\in\mathbb{N}$, let $x_n$ be a closest-point projection of $y_0$ to $[F_n^-,F_n^+]$ at which $[F_n^-,F_n^+]$ is not $(n;B_0)$-progressing. Then up to passing to a subsequence, we can assume that the sequence $(x_n)_{n\in\mathbb{N}}$ converges to a point $x\in\mathcal{T}(S)$, and $x$ is a closest-point projection of $y_0$ to the Teichmüller line from $F^-$ to $F^+$.

Let $K_1>0$ be such that for all $y,y'\in\mathcal{T}(S)$, we have $$d_{\mathcal{C}}(\pi(y),\pi(y'))\le K_1d_{\mathcal{T}}(y,y')+K_1.$$ Let $z\in\mathcal{T}(S)$ be a point lying on $[F^-,F^+]$, to the right of $x$, such that $d_{\mathcal{C}}(\pi(x),\pi(z))\ge B_0+4K_1$. Let $(z_n)_{n\in\mathbb{N}}\in\mathcal{T}(S)^{\mathbb{N}}$ be a sequence converging to $z$, with $z_n$ lying on $[F_n^-,F_n^+]$ for all $n\in\mathbb{N}$. For all $n\in\mathbb{N}$, since $[F_n^-,F_n^+]$ is not $(n;B_0)$-progressing at $x_n$, we have $d_{\mathcal{C}}(\pi(x_n),\pi(z_n))\le B_0$ for all sufficiently large $n\in\mathbb{N}$. We choose $n\in\mathbb{N}$ sufficiently large so that we also have $d_{\mathcal{T}}(x_n,x)<\frac{1}{2}$ and $d_{\mathcal{T}}(z_n,z)<\frac{1}{2}$. Then $d_{\mathcal{C}}(\pi(x_n),\pi(x))<2K_1$ and $d_{\mathcal{C}}(\pi(z_n),\pi(z))<2K_1$. These inequalities lead to a contradiction. 
\end{proof}

\begin{proof}[Proof of Proposition \ref{density-2}]
Let $C$ be the map provided by Lemma \ref{C-mod}, and let 
\begin{displaymath}
\begin{array}{cccc}
\Psi:&\overline{\mathcal{P}}&\to & \mathbb{R}\\
&\mathbf{\Phi}&\mapsto & C(\text{bnd}^-(\mathbf{\Phi}),\text{bnd}^+(\mathbf{\Phi}))
\end{array}
\end{displaymath}
\noindent (where we recall that $\overline{\mathcal{P}}$ denotes the space of bilateral sample paths of the random walk on $(\text{Mod}(S),\mu)$). Lemma \ref{C-mod} implies that $\Psi$ is a Borel map  
 and $\overline{\mathbb{P}}$-a.e. finite. Let $U$ be the transformation of $\overline{\mathcal{P}}$ induced by the Bernoulli shift on the space of bilateral sequences of increments. Then for all $k\in\mathbb{Z}$ and all $\mathbf{\Phi}\in\overline{\mathcal{P}}$, we have $\text{bnd}^-(U^k.\mathbf{\Phi})=\Phi_k^{-1}\text{bnd}^-(\mathbf{\Phi})$ and $\text{bnd}^+(U^k.\mathbf{\Phi})=\Phi_k^{-1}\text{bnd}^+(\mathbf{\Phi})$. Hence all subsegments of $d_{\mathcal{T}}$-length $\Psi(U^k.\mathbf{\Phi})$ starting at a closest-point projection of $\Phi_k.y_0$ to the Teichmüller geodesic $[\text{bnd}^-(\mathbf{\Phi}),\text{bnd}^+(\mathbf{\Phi})]$ project under $\pi$ to a region of diameter at most $B_0$ in $\mathcal{C}(S)$. As $\Psi$ is a Borel map  and $\overline{\mathbb{P}}$-a.e. finite, we can choose $M>0$ such that $$\overline{\mathbb{P}}(\Psi(\mathbf{\Phi})\le M)>\theta.$$ An application of Birkhoff's ergodic theorem to the ergodic transformation $U$ then implies that for $\overline{\mathbb{P}}$-a.e. bilateral sample path $\mathbf{\Phi}:=(\Phi_n)_{n\in\mathbb{Z}}$ of the random walk on $(\text{Mod}(S),\mu)$, if $\gamma$ denotes the Teichmüller line joining $\text{bnd}^-(\mathbf{\Phi})$ to $\text{bnd}^+(\mathbf{\Phi})$ (which almost surely exists because the hitting measures are nonatomic and concentrated on $\mathcal{UE}$), then the density of times $k\in\mathbb{Z}$ such that all subsegments of $\gamma$ of $d_{\mathcal{T}}$-length $M$    
starting at a closest-point projection of $\Phi_k.y_0$ has a $\pi$-image of diameter at least $B_0$ in $\mathcal{C}(S)$, is at least $\theta$. 
\end{proof}

\begin{prop}\label{density}
Let $\mu$ be a nonelementary probability measure on $\text{Mod}(S)$. Then for all $\theta\in (0,1)$, there exist $B,D>0$ such that for $\overline{\mathbb{P}}$-a.e. bilateral sample path $\mathbf{\Phi}:=(\Phi_n)_{n\in\mathbb{Z}}$ of the random walk on $(\text{Mod}(S),\mu)$, the set of integers $k\in\mathbb{Z}$ such that $[\text{bnd}^-(\mathbf{\Phi}),\text{bnd}^+(\mathbf{\Phi})]$ is $(B,D)$-contracting at all $d_{\mathcal{T}}$-closest-point projections of $\Phi_k.y_0$ to $[\text{bnd}^-(\mathbf{\Phi}),\text{bnd}^+(\mathbf{\Phi})]$, has density at least $\theta$.  
\end{prop}

\begin{proof}
This follows from Propositions \ref{contraction} and \ref{density-2}.
\end{proof}

\subsection{End of the proof of the spectral theorem}\label{sec3-4}

\begin{proof}[Proof of Theorem \ref{spectral}]
We will check that $(\text{Mod}(S),\mu)$ satisfies the hypotheses in Theorem \ref{spectral-gen}. The pair $(\mathcal{T}(S),\mathcal{C}(S))$ is a hyperbolic $\text{Mod}(S)$-electrification (Proposition \ref{elec-mod}). All elements of $\text{Mod}(S)$ acting loxodromically on $\mathcal{C}(S)$ (i.e. pseudo-Anosov mapping classes) have an electric translation axis in $\mathcal{T}(S)$. In addition $(\text{Mod}(S),\mu)$ is $\overline{\mathcal{T}(S)}$-boundary converging (Theorem \ref{KM}). The first item in the definition of contracting pencils of geodesics (Definition \ref{pencil}) was established by Tiozzo \cite[Lemma 19]{Tio13} for the collection of Teichmüller lines between any two transverse measured foliations. In view of Proposition \ref{density}, the second item in this definition is also satisfied by this collection of lines (notice that all points in $\mathcal{T}(S)$ are $2$-high since the Teichmüller metric on $\mathcal{T}(S)$ is symmetric). Therefore, Theorem \ref{spectral-gen} applies to $(\text{Mod}(S),\mu)$. Since the translation length of any pseudo-Anosov mapping class $\Phi$ on $\mathcal{T}(S)$ is equal to the logarithm of its stretching factor, and $\lambda(\Phi)=\lambda(\Phi^{-1})$, we get that for $\mathbb{P}$-a.e. sample path of the random walk on $(\text{Mod}(S),\mu)$, we have $$\lim_{n\to +\infty}\frac{1}{n}\log\lambda(\Phi_n)=L_{\mathcal{T}},$$ as claimed.
\end{proof}

\section{Spectral theorem for random walks on $\text{Out}(F_N)$}\label{sec4}

Let $N\ge 3$, let $F_N$ be a finitely generated free group of rank $N$, and let $\text{Out}(F_N)$ be its outer automorphism group. The goal of this section is to prove the following spectral theorem for random walks on $\text{Out}(F_N)$. We refer to Section \ref{sec-out-background} below for definitions.

\begin{theo}\label{theo-spectral-out}
Let $\mu$ be a nonelementary probability measure on $\text{Out}(F_N)$ with finite support. Then for $\mathbb{P}$-a.e. sample path $(\Phi_n)_{n\in\mathbb{N}}$ of the random walk on $(\text{Out}(F_N),\mu)$, we have $$\lim_{n\to +\infty}\frac{1}{n}\log(\lambda(\Phi_n^{-1}))=L,$$ where $\lambda(\Phi_n^{-1})$ is the stretching factor of the automorphism $\Phi_n^{-1}$, and $L$ is the drift of the random walk with respect to the Lipschitz metric on outer space.
\end{theo}

\begin{rk}\label{rk-out}
We recall that when $\mu$ has finite support, then with probability equal to $1$, eventually all automorphisms $\Phi_n$ are fully irreducible, and therefore the stretching factor $\lambda(\Phi_n^{-1})$ is well-defined. Theorem \ref{theo-spectral-out} will be proved by applying Theorem \ref{spectral-gen} to the hyperbolic $\text{Out}(F_N)$-electrification made of Culler--Vogtmann's outer space $CV_N$ and the free factor graph $FF_N$. Again, in view of Remark \ref{rk-abstract}, if $\mu$ is only assumed to have finite second moment with respect to the Lipschitz metric on $CV_N$, then Theorem \ref{theo-spectral-out} remains true. If $\mu$ is only assumed to have finite first moment with respect to the Lipschitz metric on $CV_N$, then one has to replace the limit by a limsup. 
\end{rk}

\begin{rk}
In the context of automorphisms of free groups, Kapovich also defined in \cite{Kap06} a notion of \emph{generic stretching factor} $\lambda_{A}(\Phi)$ of an outer automorphism $\Phi\in\text{Out}(F_N)$ with respect to a free basis $A$ of $F_N$, which measures the growth under $\Phi$ of an element obtained along a typical trajectory of the nonbacktracking simple random walk on $F_N$ with respect to the generating set $A$. Generic stretching factors of random outer automorphisms of $F_N$ also satisfy the same spectral property as the one established in Theorem \ref{theo-spectral-out}, namely $$\lim_{n\to +\infty}\frac{1}{n}\log(\lambda_{A}(\Phi_n^{-1}))=L,$$ as soon as the measure $\mu$ has finite first moment with respect to the Lipschitz metric on $CV_N$ (the proof of this fact does not require to apply the arguments from the present paper: this follows from \cite[Theorem 1.6]{KL15}, for instance). 
\end{rk}

The section is organized as follows. In Section \ref{sec-out-background}, we review background on $\text{Out}(F_N)$ and spaces it acts on. We then establish one more property we will need concerning the geometry of the free factor graph in Section \ref{sec4-3}. The next three sections are devoted to proving Theorem \ref{theo-spectral-out} by applying Theorem \ref{spectral-gen} to the actions of $\text{Out}(F_N)$ on Culler--Vogtmann's outer space, and on the free factor graph. In Section \ref{sec-contr-out}, we establish a sufficient condition for a subsegment of a folding line to satisfy the required contraction property, and we then check in Section \ref{sec4-5} that typical folding lines contain infinitely many subsegments satisying this condition. The proof of Theorem \ref{theo-spectral-out} is completed in Section \ref{sec4-6}.

\subsection{General background on $\text{Out}(F_N)$}\label{sec-out-background}

The group $\text{Out}(F_N)$ acts on Culler--Vogtmann's outer space $CV_N$ and on the free factor graph $FF_N$. We start by reviewing background about these spaces.

\paragraph*{Outer space and its metric.} \emph{Outer space} $CV_N$ is the collection of equivalence classes of free, simplicial, minimal and isometric actions of $F_N$ on simplicial metric trees \cite{CV86}. Two trees $T$ and $T'$ are equivalent whenever there is an $F_N$-equivariant homothety from $T$ to $T'$. \emph{Unparameterized outer space} $cv_N$ is defined by considering trees up to $F_N$-equivariant isometry, instead of homothety. The group $\text{Out}(F_N)$ acts on $CV_N$ and $cv_N$ on the right by precomposing the actions (one can also consider a left action by setting $\Phi.T:=T.\Phi^{-1}$). Outer space may be viewed as a simplicial complex with missing faces. The \emph{simplex} $\Delta(S)$ of a tree $S$ is the subspace of $CV_N$ obtained by making all edge lengths of $S$ vary (we allow some edges to have length $0$, as soon as the corresponding collapse map does not change the isotopy type of the graph). Given $\epsilon>0$, the \emph{$\epsilon$-thick part} $CV_N^{\epsilon}$ is defined as the subspace of $CV_N$ made of those trees $T\in CV_N$ such that the volume one representative of $T/F_N$ contains no embedded loop of length smaller than $\epsilon$.

Given a tree $T\in CV_N$ and an element $g\in F_N$, the \emph{translation length} of $g$ in $T$ is defined as $$||g||_T:=\inf_{x\in T}d_T(x,gx).$$ Culler--Morgan have compactified outer space \cite{CM87} by taking the closure of the image of the map 
\begin{displaymath}
\begin{array}{cccc}
i:&CV_N &\to & \mathbb{PR}^{F_N}\\
& g&\mapsto & (||g||_T)_{g\in F_N}
\end{array}
\end{displaymath}
\noindent in the projective space $\mathbb{PR}^{F_N}$. This compactification $\overline{CV_N}$ was identified by Cohen--Lustig \cite{CL95} and Bestvina--Feighn \cite{BF94} (see also \cite{Hor14-5}) with the space of equivariant homothety classes of minimal, \emph{very small} $F_N$-trees (i.e. trees whose arc stabilizers are either trivial, or maximally cyclic, and whose tripod stabilizers are trivial).

Outer space is equipped with a natural asymmetric metric for which the action of $\text{Out}(F_N)$ is by isometries: the distance $d_{CV_N}(T,T')$ between two trees $T,T'\in CV_N$ is defined as the logarithm of the infimal Lipschitz constant of an $F_N$-equivariant map from the covolume $1$ representative of $T$ to the covolume $1$ representative of $T'$, see \cite{FM11}. White has proved (see \cite[Proposition 3.15]{FM11} or \cite[Proposition 2.3]{AK11}) that for all $F_N$-trees $T$,$T'\in CV_N$, we also have $$d_{CV_N}(T,T')=\log \sup_{g\in F_N\smallsetminus\{e\}}\frac{||g||_{T'}}{||g||_T},$$ where $T$ and $T'$ are identified with their covolume $1$ representatives. Furthermore, the supremum is achieved for an element $g\in F_N\smallsetminus\{e\}$ that belongs to a finite set $\text{Cand}(T)$ only depending on $T$, made of primitive elements of $F_N$, called \emph{candidates} in $T$, whose translation length in the covolume $1$ representative of $T$ is at most $4$.

The Lipschitz metric on $CV_N$ is not symmetric. However, the following result follows from work by Algom-Kfir and Bestvina. 

\begin{prop}(Algom-Kfir--Bestvina \cite{AKB12})\label{high}
For all $\epsilon>0$, there exists $M>0$ such that all trees in $CV_N^{\epsilon}$ are $M$-high.
\end{prop}

The Lipschitz metric on outer space is geodesic: any two points in outer space are joined by a (nonunique) geodesic segment, that can be obtained \cite[Theorem 5.6]{FM11} as a concatenation of a segment contained in a simplex, and a greedy folding path, as defined in the following paragraph.

\paragraph*{Folding lines.}  
Let $T$ and $T'$ be two $\mathbb{R}$-trees. A \emph{morphism} from $T$ to $T'$ is a map $f:T\to T'$ such that every segment of $T$ can be subdivided into finitely many subsegments, in restriction to which $f$ is an isometry. A \emph{direction} at a point $x\in T$ is a connected component of $T\smallsetminus\{x\}$. A \emph{train track structure} on $T$ is a partition of the set of directions at each point $x\in T$. Elements of the partition are called \emph{gates} at $x$. A pair $(d,d')$ of directions at $x$ is \emph{illegal} if $d$ and $d'$ belong to the same gate. Any morphism $f:T\to T'$ determines a train track structure on $T$, two directions $d,d'$ at a point $x\in T$ being in the same class of the partition if there are intervals $(x,a]\subseteq d$ and $(x,b]\subseteq d'$ which have the same $f$-image. Given $T\in cv_N$ and $T'\in\overline{cv_N}$, and a map $f:T\to T'$ which is linear on edges of $T$, we say that $f$ is \emph{optimal} if it realizes the infimal Lipschitz constant of an $F_N$-equivariant map from $T$ to $T'$, and there are at least two gates at every point in $T$ for the train track structure defined by $f$.

An \emph{optimal folding path} in unprojectivized outer space is a continuous map $\gamma:I\to cv_N$, where $I$ is an interval of $\mathbb{R}$, together with a collection of $F_N$-equivariant optimal morphisms  $f_{t,t'}:\gamma(t)\to \gamma(t')$ for all $t<t'$ in $I$, such that for all $t<t'<t''$, one has $f_{t,t''}=f_{t',t''}\circ f_{t,t'}$. We say that $\gamma$ is a \emph{greedy folding path} if in addition, for all $t_0\in\mathbb{R}$, there exists $\epsilon>0$ such that for all $t\in [t_0,t_0+\epsilon]$, the tree $\gamma(t)$ is obtained from $\gamma(t_0)$ by folding small segments of length $\epsilon$ at all illegal turns of $\gamma(t)$. The projection to $CV_N$ of a (greedy) folding path in $cv_N$ will again be called a (greedy) folding path. 

A train track structure on a tree $T\in CV_N$ is \emph{recurrent} if for every edge $e\subseteq T$, there exists a legal segment in $T$ that crosses $e$, whose projection to $T/F_N$ is a (possibly non-embedded) loop. Bestvina and Reynolds have proved in \cite[Lemma 6.9]{BR13} that for all $T\in CV_N$, and all simplices $\Delta\subseteq CV_N$, there exists a tree $T_0\in\Delta$ such that there exists an optimal morphism from a tree homothetic to $T_0$ to $T$, and any such morphism induces a recurrent train track structure on $T_0$. As noticed in \cite[Remark 6.10]{BR13}, this continues to hold if $T\in\partial CV_N\smallsetminus\overline{\Delta}$. In this situation, we will say that $T_0$ is \emph{fully recurrent} with respect to $T$.

\paragraph*{The free factor graph and its Gromov boundary.} A \emph{free factor} of $F_N$ is a subgroup $A\subseteq F_N$ such that there exists a subgroup $B\subseteq F_N$ with $F_N=A\ast B$. The \emph{free factor graph} $FF_N$ is the graph whose vertices are the conjugacy classes of proper free factors of $F_N$, two vertices $[A]$ and $[B]$ being joined by an edge if there are representatives $A$ and $B$ in their conjugacy classes so that either $A\subsetneq B$ or $B\subsetneq A$. The group $\text{Out}(F_N)$ has a natural left action on $FF_N$. The free factor graph was proven to be Gromov hyperbolic by Bestvina and Feighn \cite{BF12}. One can define a coarsely $\text{Out}(F_N)$-equivariant (with respect to the left actions on both $CV_N$ and $FF_N$), coarsely Lipschitz map $\pi:CV_N\to FF_N$, by sending any tree $T$ to the conjugacy class of a proper free factor that is elliptic in a simplicial tree obtained by equivariantly collapsing some edges in $T$ to points. Bestvina and Feighn also proved that $\pi$-images in $FF_N$ of greedy folding lines in $CV_N$ are uniform unparameterized quasi-geodesics. 

\indent The Gromov boundary of the free factor graph was described by Bestvina--Reynolds \cite{BR13} and independently Hamenstädt \cite{Ham12}. Given a tree $T\in\overline{CV_N}$ and a subgroup $A\subseteq F_N$ which does not fix any point in $T$, there exists a unique minimal nonempty $A$-invariant subtree of $T$, which is called the \emph{$A$-minimal subtree} of $T$. A tree $T\in\partial{CV_N}$ is \emph{arational} if no proper free factor of $F_N$ is elliptic in $T$, and all proper free factors of $F_N$ act freely and simplicially on their minimal subtree in $T$. We denote by $\mathcal{AT}$ the subspace of $\overline{CV_N}$ made of arational $F_N$-trees.

\begin{theo}(Bestvina--Reynolds \cite{BR13}, Hamenstädt \cite{Ham12})\label{bdy}
There exist $M\in\mathbb{R}$ and a map $$\psi:\overline{CV_N}\to FF_N\cup\partial_{\infty}FF_N$$ which agrees with $\pi$ on $CV_N$, and such that 
\begin{itemize}
\item for all $T\in\mathcal{AT}$, and all sequences $(T_n)_{n\in\mathbb{N}}\in CV_N^{\mathbb{N}}$ converging to $T$, the sequence $(\pi(T_n))_{n\in\mathbb{N}}$ converges to $\psi(T)$, and 
\item \cite[Lemma 6.16]{BR13} for all $T\in\overline{CV_N}\smallsetminus\mathcal{AT}$, and all greedy folding lines $\gamma:\mathbb{R}_+\to \overline{CV_N}$, if $\gamma(t)$ converges to $T$ as $t$ goes to $+\infty$, then $\psi(\gamma(t))$ eventually stays at distance at most $M$ from $\psi(T)$.
\end{itemize} 
\end{theo}

We will fix once and for all a map $\psi$ provided by Theorem \ref{bdy}. We denote by $\mathcal{UE}$ the subspace of $\overline{CV_N}$ consisting of trees $T\in\mathcal{AT}$ such that $\psi^{-1}(\{\psi(T)\})=T$ -- this notation is used to mimic the corresponding notation for uniquely ergodic foliations in the mapping class group case; more accurately, trees in $\mathcal{UE}$ are precisely the \emph{uniquely ergometric} trees in the sense of \cite{CHL07}. We denote by $\Delta$ the diagonal subset of $\mathcal{UE}\times\mathcal{UE}$. We mention the following property of greedy folding paths, which follows from work by Bestvina--Reynolds.

\begin{prop}(Bestvina--Reynolds \cite[Corollaries 6.7 and 6.8, Lemma 6.11]{BR13})\label{cv-out}
Let $(T,T')\in\mathcal{UE}\times\mathcal{UE}\smallsetminus\Delta$. Let $((T_n,T'_n))_{n\in\mathbb{N}}\in (\mathcal{UE}\times\mathcal{UE}\smallsetminus\Delta)^{\mathbb{N}}$ be a sequence that converges to $(T,T')$. For all $n\in\mathbb{N}$, let $\gamma_n:\mathbb{R}\to CV_N$ be a greedy folding line from $T_n$ to $T'_n$. Then there exist translations $\tau_n: \mathbb{R}\to \mathbb{R}$ such that the sequence $(\gamma_n\circ \tau_n)_{n\in\mathbb{N}}$ converges uniformly on compact sets to a greedy folding line from $T$ to $T'$.
\end{prop}

\paragraph*{Projections to folding paths.} 
In their proof of the hyperbolicity of the free factor graph, Bestvina and Feighn introduced a notion of projection of a free factor to a greedy folding path in the following way. Let $I\subseteq\mathbb{R}$ be an interval, and let $\gamma:I\to CV_N$ be a greedy folding path in $CV_N$ determined by a morphism $f$. The morphism $f$ determines a train track structure on all trees $\gamma(t)$ with $t\in I$. Given a primitive element $g\in F_N$, we let $\text{right}_{\gamma}(g)$ be the supremum of all times $t\in I$ such that the axis of $g$ in the covolume $1$ representative of $\gamma(t)$ contains a segment of length $M_{BF}$ that does not contain any legal subsegment of length $3$, where $M_{BF}$ is a fixed constant that only depends on the rank $N$ of the free group, see \cite[Section 6]{BF12}. As follows from the work by Bestvina--Feighn, the projection satisfies the following property. 

\begin{lemma}(Bestvina--Feighn \cite{BF12})\label{bf}
There exists $K_0>0$ (only depending on the rank $N$ of the free group) such that the following holds. Let $\gamma:I\to CV_N$ be a greedy folding path, let $g\in F_N$ be a primitive element, and let $t,t'\ge\text{right}_{\gamma}(g)$  
 with $t\le t'$. Then $$\left|d_{CV_N}(\gamma(t),\gamma(t'))-\log\frac{||g||_{\gamma(t')}}{||g||_{\gamma(t)}}\right|\le K_0.$$
\end{lemma}

\begin{proof}
Let $\sigma$ be a loop obtained by projecting a fundamental domain of the axis of $g$ to the volume $1$ representative of the quotient graph $\gamma(t)/F_N$. Then all subsegments of length $M_{BF}$ in $\sigma$ contain a legal subsegment of length $3$ for the train-track structure on $\gamma(t)$. It follows from the estimate in \cite[Lemma 4.4]{BF12} that $\sigma$ contains $\left\lfloor\frac{1}{M_{BF}}||g||_{\gamma(t)}\right\rfloor$ legal segments of length $3$, whose middle subsegments of length $1$ all grow exponentially fast with maximal speed along the folding path $\gamma$, and never get identified with one another during the folding procedure: exponential growth is stated in \cite[Corollary 4.8]{BF12}, and the fact that the segments never get identified also follows from Bestvina--Feighn's arguments because these segments grow faster than their extremities get folded. The required estimate follows from this observation.
\end{proof}

Given a tree $S\in CV_N$, we let $$\text{right}_{\gamma}(S):=\sup_{g\in\text{Cand}(S)}\text{right}_{\gamma}(g).$$ We then let $\text{Pr}_{\gamma}(S):=\gamma(\text{right}_{\gamma}(S))$ (we will sometimes also use the notation $\text{Pr}_I(S)$ when it turns out to be more convenient). The following contraction property for projections of greedy folding paths to $FF_N$ was established by Bestvina and Feighn. This is stated in \cite{BF12} with a slightly different notion of projection (namely, the left projection instead of the right projection), however it follows from \cite[Proposition 6.4]{BF12} that this other projection lies at bounded distance apart from the one we consider.

\begin{lemma}(Bestvina--Feighn \cite[Proposition 7.2]{BF12})\label{BF-contr}
There exists $D_1>0$ such that for all greedy folding paths $\gamma:I\to CV_N$ (where $I\subseteq\mathbb{R}$ is an interval), and all $H,H'\in CV_N$, if $d_{CV_N}(H,H')\le d_{CV_N}(H,\text{Im}(\gamma))$, then $d_{FF_N}(\pi(\text{Pr}_{\gamma}(H)),\pi(\text{Pr}_{\gamma}(H')))\le D_1$.
\end{lemma}

The following lemma, due to Dowdall and Taylor, relates the Bestvina--Feighn projection to the closest-point projection in $FF_N$. Given a greedy folding path $\gamma:I\to CV_N$, we should denote by $\mathbf{n}_{\pi\circ\gamma}$ the closest-point projection to the image of $\pi\circ\gamma$ in $FF_N$. However,  this is a priori a subset, and it is sometimes convenient to have  chosen a single point in it, thus we rather define  $\mathbf{n}_{\pi\circ\gamma}$  to be the leftmost  closest-point projection in the sense that, in the parametrisation of $ \pi\circ\gamma $ as a path, it has smaller parameter value among the closest point projections.

\begin{lemma}(Dowdall--Taylor \cite[Lemma 4.2]{DT14})\label{DT}
There exists a constant $D_2>0$ such that for any $H\in CV_N$, and any greedy folding path $\gamma:I\to CV_N$, we have $$d_{FF_N}(\pi(\text{Pr}_{\gamma}(H)),\mathbf{n}_{\pi\circ\gamma}(\pi(H)))\le D_2.$$
\end{lemma}

\paragraph*{Fully irreducible automorphisms.} An element $\Phi\in\text{Out}(F_N)$ is \emph{fully irreducible} if no power $\Phi^k$ with $k\neq 0$ fixes the conjugacy class of a proper free factor of $F_N$. Any fully irreducible automorphism of $F_N$ has a well-defined \emph{stretching factor} $\lambda(\Phi)$, which satisfies that for all primitive elements $g\in F_N$, we have $$\lim_{n\to +\infty}\sqrt[n]{||\Phi^n(g)||}=\lambda(\Phi),$$ where $||\Phi^n(g)||$ denotes the smallest word length of a conjugate of $\Phi^n(g)$, written in some fixed free basis of $F_N$. Bestvina and Feighn have shown \cite{BF12}  
 that an element $\Phi\in\text{Out}(F_N)$ acts as a loxodromic isometry of $FF_N$ if and only if $\Phi$ is fully irreducible. Any fully irreducible element $\Phi\in\text{Out}(F_N)$ has a translation axis in $CV_N$, which is a folding line \cite{HM11}. The amount of translation is equal to $\log(\lambda(\Phi))$. With the terminology from Section \ref{abstract}, all elements in $\text{Out}(F_N)$ acting as fully irreducible isometries of $FF_N$ have an electric axis. 

\paragraph*{Random walks on $\text{Out}(F_N)$.} We finally record a result of Namazi--Pettet--Reynolds concerning random walks on $\text{Out}(F_N)$. A subgroup $H\subseteq\text{Out}(F_N)$ is \emph{nonelementary} if $H$ contains two fully irreducible automorphisms that generate a free subgroup of $\text{Out}(F_N)$: with the terminology from Section \ref{sec1}, this is equivalent to nonelementarity of the $H$-action on $FF_N$. We say that a probability measure on $\text{Out}(F_N)$ is \emph{nonelementary} if the subsemigroup generated by the support of $\mu$ is a nonelementary subgroup of $\text{Out}(F_N)$. With the terminology from Section \ref{abstract}, the following theorem asserts that $(\text{Out}(F_N),\mu)$ is $\overline{CV_N}$-boundary converging.

\begin{theo}(Namazi--Pettet--Reynolds \cite{NPR14})\label{NPR}
Let $\mu$ be a nonelementary probability measure on $\text{Out}(F_N)$ with finite first moment with respect to $d_{CV_N}$. Then for $\mathbb{P}$-a.e. sample path $\mathbf{\Phi}:=(\Phi_n)_{n\in\mathbb{N}}$ of the random walk on $(\text{Out}(F_N),\mu)$, and any $y_0\in CV_N$, the sequence $(\Phi_n.y_0)_{n\in\mathbb{N}}$ converges to a point $\text{bnd}(\mathbf{\Phi})\in\mathcal{UE}$. The hitting measure $\nu$ defined by setting $$\nu(S):=\mathbb{P}(\text{bnd}(\mathbf{\Phi})\in S)$$ for all measurable subsets $S\subseteq \partial CV_N$ is nonatomic. 
\end{theo}

\subsection{More on the geometry of the free factor graph}\label{sec4-3}

We will need to establish the following proposition concerning the geometry of the projection of folding paths to the free factor graph. We refer to the statement of Theorem \ref{bdy} above for the definition of the map $\psi$.

\begin{prop}\label{ff-geom}
There exists $M>0$ such that for all sufficiently large $D>0$, all $y_0,y_1\in CV_N$, all trees $T\in\overline{CV_N}$ and all sequences $(T_n)_{n\in\mathbb{N}}\in CV_N^{\mathbb{N}}$ converging to $T$, if $(\pi(T_n)|\pi(y_1))_{\pi(y_0)}\ge D$ for all $n\in\mathbb{N}$, then $(\psi(T)|\pi(y_1))_{\pi(y_0)}\ge D-M$. 
\end{prop}

Proposition \ref{ff-geom} follows from Theorem \ref{bdy} in the particular case where all trees in the sequence $(T_n)_{n\in\mathbb{N}}$ lie on the image of a greedy folding path in $CV_N$. We need to extend this to arbitrary sequences. Our proof of Proposition \ref{ff-geom} will make use of yet another $\text{Out}(F_N)$-graph, namely the \emph{cyclic splitting graph} $FZ_N$. This is the graph whose vertices correspond to the splittings of $F_N$ as a graph of groups with cyclic (possibly trivial) edge groups, two splittings being joined by an edge if they have a common refinement. Hyperbolicity of $FZ_N$ was established by Mann \cite{Man12}, and its geometry was further studied in \cite{Hor14-6}. There is a coarsely $\text{Out}(F_N)$-equivariant map $\Theta$  
 from $FZ_N$ to $FF_N$: given any cyclic splitting $T$ of $F_N$, there exists a proper free factor of $F_N$ which is elliptic in a tree obtained by equivariantly collapsing some edges of $T$, and $\Theta$ maps $T$ to one of these free factors. We will denote by $\pi_Z$ the natural map from $CV_N$ to $FZ_N$. Images in $FZ_N$ of greedy folding paths in $CV_N$ are uniform unparameterized quasigeodesics \cite{Man12}. 

\begin{proof}
Let $f$ be an optimal morphism  
from a tree whose homothety class belongs to the closure $\overline{\Delta(y_0)}$ of the simplex of $y_0$, to $T$. Let $(f_n)_{n\in\mathbb{N}}$ be a sequence of optimal morphisms which converges to $f$ for the Gromov--Hausdorff topology on morphisms introduced in \cite[Section 3.2]{GL07}, such that for all $n\in\mathbb{N}$, the map $f_n$ is a morphism from a tree $y'_n$ whose homothety class belongs to $\Delta(y_0)$, to $T_n$. Let $\gamma:\mathbb{R}_+\to \overline{CV_N}$ be the greedy folding path determined by $f$, and for all $n\in\mathbb{N}$, let $\gamma_n:\mathbb{R}_+\to CV_N$ be the greedy folding path from $y'_n$ to $T_n$ determined by $f_n$. By \cite[Proposition 2.4]{GL07}, the folding paths $\gamma_n$ converge uniformly on compact sets to $\gamma$.

There exists a constant $M_1>0$, only depending on $K_2$ and the hyperbolicity constant of $FF_N$, such that the following holds. Let $y_2\in CV_N$ be such that $\pi(y_2)$ lies at bounded distance apart from a geodesic segment from $\pi(y_0)$ to $\pi(y_1)$ in $FF_N$, and $d_{FF_N}(\pi(y_0),\pi(y_2))=D-M_1$. Then for all $n\in\mathbb{N}$, the $\pi$-images in $FF_N$ of all greedy folding paths $\gamma_n$ pass through the $M_1$-neighborhood $U_1$ of $\pi(y_2)$. 

Recall that greedy folding paths in $CV_N$ project to uniform unparameterized quasigeodesics in both complexes $FZ_N$ and $FF_N$. Therefore, there exists $M_2>0$ (only depending on the hyperbolicity constants of $FF_N$ and $FZ_N$), a bounded neighborhood $U_2$ of $\pi(y_2)$ in $FF_N$, and a tree $S\in CV_N$ whose $\pi$-image in $FF_N$ lies in $U_2$, such that for all $n\in\mathbb{N}$, one can find a tree $S_n$ lying on the image of $\gamma_n$, whose projection to $FZ_N$ lies at distance at most $M_2$ from $\pi_Z(S)$.

Let $S_{\infty}\in\overline{CV_N}$ be a limit point of the sequence $(S_n)_{n\in\mathbb{N}}$. By \cite[Proposition 2.4]{GL07}, the tree $S_{\infty}$ belongs to $\gamma(\mathbb{R}_+\cup\{+\infty\})$. Since the sequence $(\psi_Z(S_n))_{n\in\mathbb{N}}$ is bounded, the tree $S_{\infty}$ is not $\mathcal{Z}$-averse in the sense of \cite{Hor14-6}. It follows from \cite[Proposition 8.8]{Hor14-6} that there exists $M_3>0$ such that the set of reducing $\mathcal{Z}$-splittings of $S_{\infty}$ (defined in \cite[Section 5.1]{Hor14-6}) lies in the $M_3$-neighborhood $U_3$ of $\pi_Z(S)$ in $FZ_N$, and the image of $\gamma$ in $FZ_N$ meets $U_3$. The projection of $U_3$ to $FF_N$ is a bounded neighborhood of $\pi(y_2)$, and the $\pi$-image of $\gamma$ meets this neighborhood.  The last item from Theorem \ref{bdy} implies that all points $\pi\circ\gamma(t)$ with $t$ sufficiently large belong to a bounded neighborhood of $\psi(T)$. The claim follows from these observations.
\end{proof}

From Proposition \ref{ff-geom}, where one chooses $y_0$ such that $\pi(y_0)$ is close to $\psi(T)$, and $y_1$ such that $\pi(y_1)$ belongs to the bounded region of $FF_N$ that contains all points $\pi(T_n)$, one deduces the following analogue for $FF_N$ of \cite[Proposition 8.8]{Hor14-6} (which was established for $FZ_N$ there).

\begin{cor}\label{fz-ff}
For all $D>0$, there exists $M>0$ such that the following holds. Let $T\in\overline{CV_N}$, and let $(T_n)_{n\in\mathbb{N}}\in CV_N^{\mathbb{N}}$ be a sequence of trees that converges to $T$. Assume that all projections $\pi(T_n)$ belong to a common region of $FF_N$ of diameter $D$. Then $T$ is not arational, and for all $n\in\mathbb{N}$, we have $d_{FF_N}(\pi(T_n),\psi(T))\le M$. 
\qed
\end{cor}

\subsection{A contraction property for folding lines in outer space}\label{sec-contr-out} 

Let $\delta$ be the hyperbolicity constant of $FF_N$. Let $K_2>0$ be a constant such that all $\pi$-images of greedy folding lines in $CV_N$ are $K_2$-unparameterized quasigeodesics in $FF_N$, and let $\kappa=\kappa(K_2,\delta)$ be the constant provided by Proposition \ref{hyp-drift}. We now introduce our definition of progress for folding paths in outer space, which will play the same role in our arguments as in the case of mapping class groups in the previous section.  This is illustrated in Figure \ref{fig-progression}.

\begin{de}\textbf{(Progress for folding paths)}\label{progression}\\
Let $(T,T')\in\mathcal{UE}\times\mathcal{UE}\smallsetminus\Delta$. Let $\gamma:\mathbb{R}\to CV_N$ be a greedy  
 folding line from $T$ to $T'$, and let $S$ be a tree lying on the image of $\gamma$. Given $C,A_0,B_0,C_0>0$, we say that $\gamma$ is \emph{$(C;A_0,B_0,C_0)$-progressing at $S$} if the following holds.\\
Let 
\begin{itemize}
\item $\widetilde{S}\in CV_N$ be a tree lying to the right of $S$ on the image of $\gamma$, satisfying the inequality $d_{FF_N}(\pi(S),\pi(\widetilde{S}))\le A_0$, and
\item $R\in CV_N$ be a tree such that $\textbf{n}_{\pi\circ\gamma}(\pi(R))$ is to the right of $\pi(S)$ and satisfies $d_{FF_N}(\pi(\widetilde{S}),\textbf{n}_{\pi\circ\gamma}(\pi(R)))\ge B_0$ for all trees $\widetilde{S}$ as above, and
\item $\gamma':[a,b]\to CV_N$ (where $[a,b]\subseteq\mathbb{R}$ is a segment) be a greedy folding path from a tree $S'$ in the simplex $\Delta(\widetilde{S})$ to $R$, where either $S'=\widetilde{S}$, or else $S'$ is fully recurrent with respect to $R$.
\end{itemize}
\noindent Then for all $c\in [a,b]$, if $\gamma'([a,c])$ has $d_{CV_N}^{sym}$-diameter at least $C$, then $\pi\circ\gamma'([a,c])$ has diameter at least $C_0$ in $FF_N$. 
\end{de}

\begin{figure}
\begin{center}
\includegraphics[scale=.6]{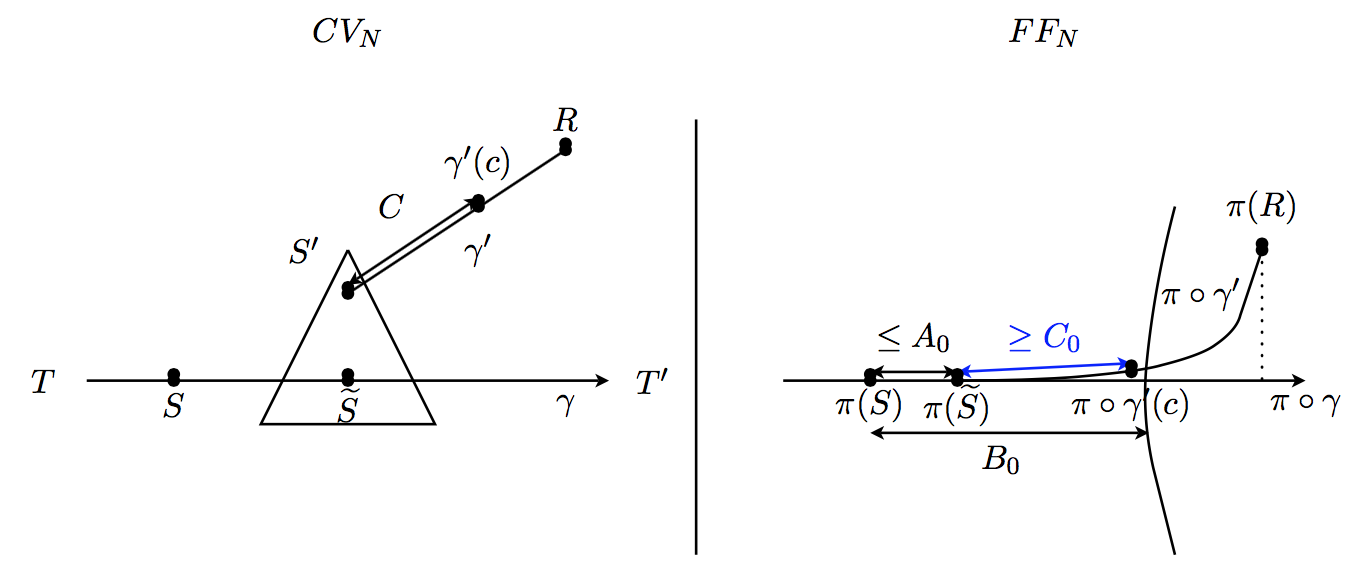}
\caption{Progress for folding paths.}
\label{fig-progression}
\end{center}
\end{figure}

\begin{rk}\label{rk-prog}
The above definition of progress implies in particular that the smallest subsegment on the image of $\gamma$ of $d_{CV_N}^{sym}$-diameter $C$ starting at $S$ has a projection of diameter at least $C_0$ in $FF_N$. This is shown by choosing $R$ to lie sufficiently to the right on the image of $\gamma$, and letting $\gamma'$ be a subpath of $\gamma$ starting at $S$. We do not know however whether Proposition \ref{contr-out} below remains valid if we took the above property as a definition of progress (in other words, we do not know whether the exact analogue of Proposition \ref{contraction} holds in the $\text{Out}(F_N)$ context). As recalled in the proof sketch of Proposition \ref{DDM}, Dowdall--Duchin--Masur's proof in the case of mapping class groups uses a distance formula relating the distance in $\mathcal{T}(S)$ to the distances in curve graphs associated to various subsurfaces of $S$. However, to our knowledge, finding an analogous distance formula in the $\text{Out}(F_N)$ context is still an open question.
\end{rk}

The goal of the present section is to prove that greedy folding paths which make progress are high and satisfy the contraction property from Definition \ref{de-contr}: this is the content of Propositions \ref{high-out} and \ref{contr-out} below. We first prove that progress implies highness, using an argument of Dowdall--Taylor \cite[Lemma 4.3]{DT14}.

\begin{prop}\label{high-out}
Let $A_0,B_0,C_0>0$. There exists $\alpha>0$ such that if $C_0>\alpha$, then the following holds.\\ For all $C>0$, there exists $M>0$ such that for all $(T,T')\in\mathcal{UE}\times\mathcal{UE}\smallsetminus\Delta$, all greedy folding lines $\gamma:\mathbb{R}\to CV_N$ from $T$ to $T'$, and all $S\in CV_N$ lying on the image of $\gamma$, if $\gamma$ is $(C;A_0,B_0,C_0)$-progressing at $S$, then $S$ is $M$-high. 
\end{prop}

\begin{proof}
In view of Proposition \ref{high}, it is enough to prove that $S$ belongs to the thick part $CV_N^{\epsilon}$ for some value of $\epsilon>0$ only depending on $C$. Let $\epsilon<\exp(-C)$. Assume towards a contradiction that $S\notin CV_N^{\epsilon}$. Then there exists $g\in F_N$ whose translation length is at most $\epsilon$ in the covolume one representative of $S$. By assumption (see Remark \ref{rk-prog} above), there exists a tree $S'$ lying on the image of $\gamma$ such that $d_{CV_N}(S,S')\le C$ and $d_{FF_N}(\pi(S),\pi(S'))\ge C_0$. If $C_0$ has been chosen sufficiently large, then $g$ cannot have translation length smaller than $1$ in the covolume one representative of $S'$. However, this implies that $$d_{CV_N}(S,S')\ge\log\frac{||g||_{S'}}{||g||_S}\ge\log\frac{1}{\epsilon}>C,$$ a contradiction. 
\end{proof}

We will now prove that progress implies contraction. The following lemma is an exercise in hyperbolic metric spaces; we recall that $\mathbf{n}_{\tau}$ and $\mathbf{n}_{\tau'}$ denote the leftmost  
closest-point projections to the corresponding quasigeodesics.

\begin{lemma}\label{haddock}
For all $K,K_2>0$, there exists $\kappa'=\kappa'(K_2,K,\delta)$ such that for all $y\in FF_N$, and all $K_2$-quasigeodesic segments $\tau:[a,b]\to FF_N$ and $\tau':[a',b']\to FF_N$, if $d_{FF_N}(\tau(a),\tau'(a'))\le K$ and $d_{FF_N}(\tau(b),\tau'(b'))\le K$, then $d_{FF_N}(\mathbf{n}_{\tau}(y),\mathbf{n}_{\tau'}(y))\le\kappa'$.  
\qed
\end{lemma}

\begin{prop}\label{contr-out}
There exist $\alpha_0,\beta_0,\gamma_0$ such that for all $A_0\ge\alpha_0$, all $B_0\ge\beta_0$ and all $C_0\ge\gamma_0$, there exists $B>0$ such that the following holds.\\
For all $C>0$, there exists $D>0$ such that for all $(T,T')\in\mathcal{UE}\times\mathcal{UE}\smallsetminus\Delta$, all greedy folding lines $\gamma:\mathbb{R}\to CV_N$ from $T$ to $T'$, and all $S\in CV_N$ lying on the image of $\gamma$, if $\gamma$ is $(C;A_0,B_0,C_0)$-progressing at $S$, then $\gamma$ is $(B,D)$-contracting at $S$. 
\end{prop}

\begin{proof}
The situation in the following proof is illustrated on Figure \ref{fig-proof}. Let $K_1>0$ be such that $$d_{FF_N}(\pi(x),\pi(y))\le K_1d_{CV_N}(x,y)+K_1$$ for all $x,y\in CV_N$.
We let $B:=B_0+K_1+2\kappa$. Let $a\in\mathbb{R}$ be such that $\gamma(a)=S$, and let $b\ge a$ be the infimum of all real numbers $b'$ such that $\pi\circ\gamma([a,b'])$ has $d_{FF_N}$-diameter at least $B$: this implies in particular that $\pi\circ\gamma([a,b])$ has $d_{FF_N}$-diameter at least $B_0+2\kappa$. We let $I:=[a,b]$, and let $S':=\gamma(b)$. Let $\gamma':\mathbb{R}\to CV_N$ be a geodesic line, whose $\pi$-image in $FF_N$ crosses $\pi\circ\gamma_{|I}$ up to distance $\kappa$. It is enough to prove that the image of $\gamma'$ in $CV_N$ passes at bounded $d_{CV_N}$-distance of $\gamma(I)$. 
\\
\\
\textbf{Step 1: Estimates coming from progress.}
\\
\\
Since $\pi\circ\gamma'$ crosses $\pi\circ\gamma_{|I}$ up to distance $\kappa$, there exists a point $U\in CV_N$ lying on the image of $\gamma'$ such that $d_{FF_N}(\pi(U),\pi(S))\le \kappa$. So $d_{FF_N}(\mathbf{n}_{\pi\circ\gamma_{|I}}(\pi(U)),\pi(S))\le 2\kappa$. Lemma \ref{DT} then implies that $d_{FF_N}(\pi(\text{Pr}_{I}(U)),\pi(S))\le 2\kappa+D_2$. 

Let $U'$ be the point lying to the right of $U$ on the image of $\gamma'$, and such that $d_{CV_N}(U,U')=d_{CV_N}(U,\gamma(I))$. Let $S_0$ be a point lying on $\gamma(I)$ such that $d_{CV_N}(U,U')=d_{CV_N}(U,S_0)$. By Lemma \ref{BF-contr}, we have $d_{FF_N}(\pi(\text{Pr}_I(U)),\pi(\text{Pr}_I(S_0)))\le D_1$. In view of Lemma \ref{DT}, we also have $d_{FF_N}(\pi(S_0),\pi(\text{Pr}_I(S_0)))\le D_2$. Using the triangle inequality, this implies that $d_{FF_N}(\pi(S_0),\pi(S))\le 2\kappa+D_1+2D_2$. Both trees $S$ and $S_0$ lie on the image of $\gamma$ (and $S_0$ lies to the right of $S$). Assume that $C_0\ge 2\kappa+D_1+2D_2$. Since $\gamma$ is $(C;A_0,B_0,C_0)$-progressing at $S$, by choosing the point $R$ from Definition \ref{progression} to lie sufficiently far apart from $S$ (see Remark \ref{rk-prog}), this implies that 
\begin{equation}\label{eqe0}
d_{CV_N}^{sym}(S_0,S)\le C.
\end{equation}
By Lemma \ref{BF-contr}, we also have $d_{FF_N}(\pi(\text{Pr}_I(U')),\pi(\text{Pr}_I(U)))\le D_1$. The triangle inequality then implies that $d_{FF_N}(\pi(\text{Pr}_I(U')),\pi(S))\le 2\kappa+D_1+D_2$. As above, this implies that 
\begin{equation}\label{e1}
d_{CV_N}^{sym}(S,\text{Pr}_I(U'))\le C.
\end{equation} 
Since $\pi\circ\gamma'$ crosses $\pi\circ\gamma_{|I}$ up to distance $\kappa$, there exists a point $U''$ lying on the image of $\gamma'$, such that $d_{FF_N}(\pi(U''),\pi(S'))\le \kappa$. In particular, we have $d_{FF_N}(\pi(S),\mathbf{n}_{\pi\circ\gamma}(\pi(U'')))\ge B-2\kappa=B_0$. Let $f$ be a morphism from a tree $\widetilde{\text{Pr}_I(U')}$ in the simplex $\Delta(\text{Pr}_{I}(U'))$ to $U''$, where $\widetilde{\text{Pr}_I(U')}$ is fully recurrent with respect to $U''$, and let $\gamma''$ be the greedy folding path determined by $f$. Notice that $d_{FF_N}(\pi(\text{Pr}_I(U)),\pi(\widetilde{\text{Pr}_I(U)}))\le K_0$ for some uniform constant $K_0$. Applying Lemma \ref{haddock} to $\tau:=\pi\circ\gamma_{|I}$ and $\tau':=\pi\circ\gamma''$ with $K:=2\kappa+D_1+D_2+K_0$, we get $d_{FF_N}(\mathbf{n}_{\pi\circ\gamma_{|I}}(\pi(U')),\mathbf{n}_{\pi\circ\gamma''}(\pi(U')))\le\kappa'$. Therefore, it follows from Lemma \ref{DT} that $d_{FF_N}(\pi(\text{Pr}_I(U')),\pi(\text{Pr}_{\gamma''}(U')))\le \kappa'+2D_2$. We also recall from the previous paragraph that $d_{FF_N}(\pi(S),\pi(\text{Pr}_I(U')))\le 2\kappa+D_1+D_2$. Assume that $A_0\ge 2\kappa+D_1+D_2$. Since $\gamma$ is $(C;A_0,B_0,C_0)$-progressing at $S$, the same argument as in the proof of Proposition \ref{high-out} implies that both $\text{Pr}_I(U')$ and $\widetilde{\text{Pr}_I(U')}$ are thick, so there exists $K>0$ (only depending on $C$ and on the rank $N$ of the free group) such that  
\begin{equation}\label{ee0}
d^{sym}_{CV_N}(\text{Pr}_I(U'),\widetilde{\text{Pr}_I(U')})\le K.
\end{equation}
Both trees $\widetilde{\text{Pr}_I(U')}$ and $\text{Pr}_{\gamma''}(U')$ lie on the image of $\gamma''$ (with $\text{Pr}_{\gamma''}(U')$ to the right of $\widetilde{\text{Pr}_I(U')}$). Assume that $C_0\ge \kappa'+2D_2$. Since $\gamma$ is $(C;A_0,B_0,C_0)$-progressing at $S$, we also have 
\begin{equation}\label{ee1}
d^{sym}_{CV_N}(\widetilde{\text{Pr}_I(U')},\text{Pr}_{\gamma''}(U'))\le C.
\end{equation}
Using Equations \eqref{ee0} and \eqref{ee1} and the triangle inequality, we obtain 
\begin{equation}\label{e2}
d^{sym}_{CV_N}(\text{Pr}_I(U'),\text{Pr}_{\gamma''}(U'))\le K+C.
\end{equation}

\noindent \textbf{Step 2: Estimates coming from the definition of the projection to a folding path.}
\\
\\
We will now prove that the length of the concatenation of the paths from $U'$ to $\text{Pr}_{\gamma''}(U')$ and from $\text{Pr}_{\gamma''}(U')$ to $U''$ represented in plain green lines on Figure \ref{fig-proof} is close to being equal to the distance from $U'$ to $U''$ (this is the meaning of Equation \eqref{eq3} below). This will imply in the next step of the proof that the total green length (both plain and dotted) is close to being equal to the distance from $U$ to $U''$, and therefore $U'$ has to be close to the image of $\gamma$. 

Let $\alpha\in F_N$ be a candidate in $U'$ (as defined in Section \ref{sec-out-background}, in the paragraphs before Proposition \ref{high}) such that 
\begin{equation}\label{1}
d_{CV_N}(U',\text{Pr}_{\gamma''}(U'))=\log\frac{||\alpha||_{\text{Pr}_{\gamma''}(U')}}{||\alpha||_{U'}}.
\end{equation}
Lemma \ref{bf} applied to the folding path $\gamma''$ implies that there exists a constant $K_0$ (only depending on the rank of the free group) such that 
\begin{equation}\label{2}
d_{CV_N}(\text{Pr}_{\gamma''}(U'),U'')\le\log\frac{||\alpha||_{U''}}{||\alpha||_{\text{Pr}_{\gamma''}(U')}}+K_0.
\end{equation}
By adding Equations \eqref{1} and \eqref{2}, we obtain that 
\begin{equation}\label{eq3}
d_{CV_N}(U',\text{Pr}_{\gamma''}(U'))+d_{CV_N}(\text{Pr}_{\gamma''}(U'),U'')-K_0\le\log\frac{||\alpha||_{U''}}{||\alpha||_{U'}}\le d_{CV_N}(U',U'').
\end{equation}
\textbf{Step 3: End of the proof.}
\\
\\
Recall that $d_{CV_N}^{sym}(\text{Pr}_I(U'),S)\le C$ (Equation \eqref{e1}) and $d^{sym}_{CV_N}(\text{Pr}_I(U'),\text{Pr}_{\gamma''}(U'))\le K+C$ (Equation \eqref{e2}). Together with Equation \eqref{eq3},     
 this implies that there exists a constant $K'$, only depending on $C$ (and on the rank of the free group), such that 
\begin{equation}\label{e3}
d_{CV_N}(U',U'')\ge d_{CV_N}(U',S)+d_{CV_N}(S,U'')-K'.
\end{equation} 
Hence
\begin{displaymath}
\begin{array}{rl}
d_{CV_N}(U',S)&\le d_{CV_N}(U',U'')-d_{CV_N}(S,U'')+K'\\
&\le d_{CV_N}(U',U'')-d_{CV_N}(U,U'')+d_{CV_N}(U,S)+K'\\
&=-d_{CV_N}(U,U')+d_{CV_N}(U,S)+K'\\
&\le -d_{CV_N}(U,U')+d_{CV_N}(U,S_0)+d_{CV_N}(S_0,S)+K'\\
&\le C+K'
\end{array}
\end{displaymath}
\noindent Indeed, the first inequality is   
 Equation \eqref{e3} and the second comes from the triangle inequality. The equality on the third line follows from the fact that $\gamma'$ is a geodesic, and that the trees $U$, $U'$ and $U''$ lie in this order on the image of $\gamma'$ (at least if $B_0$ has been chosen sufficiently large). The inequality on the fourth line follows from the triangle inequality. The last inequality uses the definition of $U'$ (which says in particular that $d_{CV_N}(U,U')=d_{CV_N}(U,S_0)$) and the fact that $d_{CV_N}^{sym}(S_0,S)\le C$ (Equation \eqref{eqe0}).
\end{proof}

\begin{figure}
\begin{center}
\includegraphics[scale=.6]{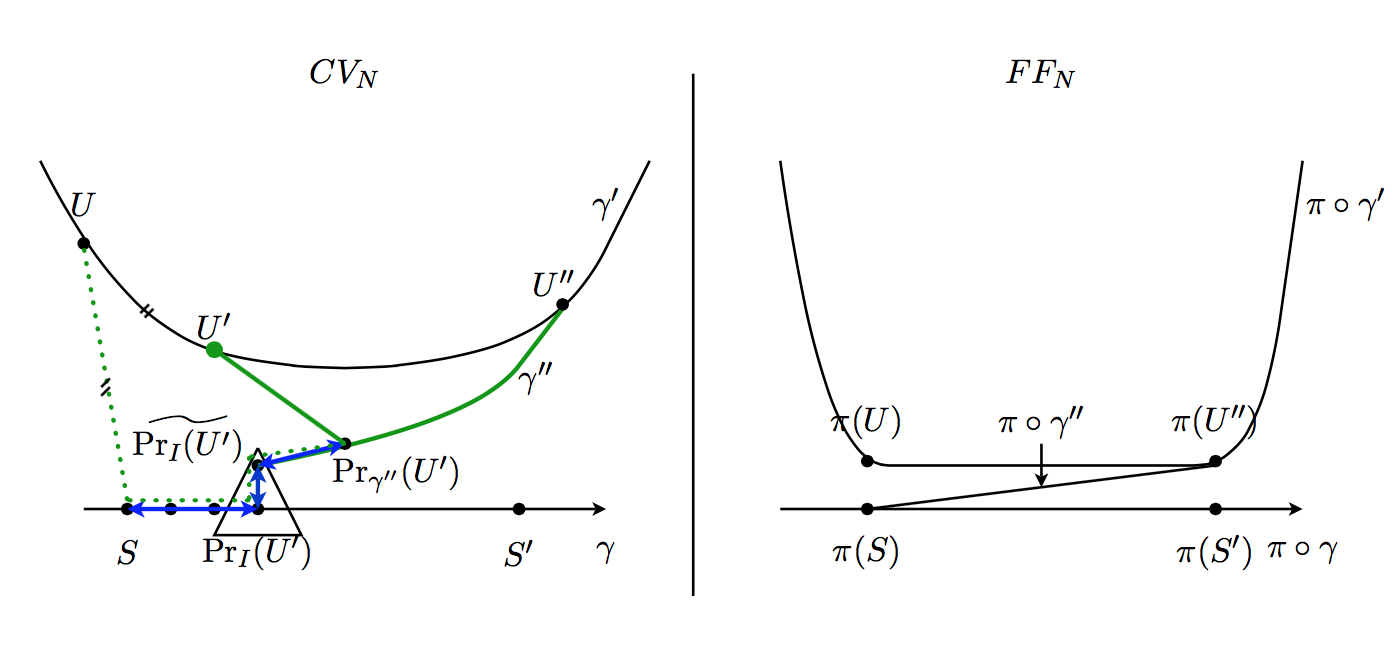}
\caption{Schematic picture of the situation in the proof of Proposition \ref{contr-out}. The blue parts all have uniformly bounded $d_{CV_N}^{sym}$-diameter.}
\label{fig-proof}
\end{center}
\end{figure}

\subsection{Typical rays are infinitely often progressing.}\label{sec4-5}

Let $(T,T')\in\mathcal{UE}\times\mathcal{UE}\smallsetminus\Delta$. For every neighborhood $\mathcal{U}$ of $(T,T')$ in $\mathcal{UE}\times\mathcal{UE}$, we let $\text{Tube}(\mathcal{U})$ be the collection of all greedy folding lines joining pairs in $\mathcal{U}$. Given a point $y_0\in CV_N$ and $A_0>0$, we then let $\text{Tube}_{y_0}(\mathcal{U};A_0)$ be the set of all trees $U\in CV_N$ such that there exists $\gamma\in\text{Tube}(\mathcal{U})$ such that 
\begin{itemize}
\item the tree $U=\gamma(t)$ belongs to the image of $\gamma$, and 
\item there exists a $d_{CV_N}^{sym}$-closest-point projection $U_0=\gamma(t_0)$ of $y_0$ to the image of $\gamma$, with $t_0\le t$, such that $\pi\circ\gamma([t_0,t])$ has diameter at most $A_0$ in $FF_N$.
\end{itemize}

\begin{lemma}\label{X_0-compact}
Let $A_0>0$,  
  $y_0\in CV_N$, and  $(T,T')\in\mathcal{UE}\times\mathcal{UE}\smallsetminus\Delta$. There exists a neighborhood $\mathcal{U}$ of $(T,T')$ in $\mathcal{UE}\times\mathcal{UE}$, such that the set $\overline{\text{Tube}_{y_0}(\mathcal{U};A_0)}$ is a compact subset of $CV_N$.   
\end{lemma}

\begin{proof}
Assume otherwise. Then we can find a sequence $((T_n,T'_n))_{n\in\mathbb{N}}\in (CV_N\times CV_N)^{\mathbb{N}}$ converging to $(T,T')$, together with greedy     
 folding paths $\gamma_n$ from $T_n$ to $T'_n$ for all $n\in\mathbb{N}$, and a sequence $(U_n)_{n\in\mathbb{N}}\in CV_N^{\mathbb{N}}$ with the following properties:
\begin{itemize}
\item for all $n\in\mathbb{N}$, the tree $U_n$ lies on the image of $\gamma_n$, and
\item for all $n\in\mathbb{N}$, there exists a $d_{CV_N}^{sym}$-closest-point projection $S_n$ of $y_0$ to $\gamma_n$, such that $U_n$ lies to the right of $S_n$, and the $\pi$-image of the subsegment of $\gamma_n$ joining $S_n$ to $U_n$ has diameter at most $A_0$ in $FF_N$, and
\item the sequence $(U_n)_{n\in\mathbb{N}}$ leaves every compact subspace of $CV_N$. 
\end{itemize}
\noindent However, in view of Proposition \ref{cv-out}, the sequence $(\gamma_n)_{n\in\mathbb{N}}$ accumulates on a folding line $\gamma$ from $T$ to $T'$. It follows from \cite[Theorem 6.6]{BR13} that the sequence $(S_n)_{n\in\mathbb{N}}$ should accumulate to a $d_{CV_N}^{sym}$-closest-point projection $S$ of $y_0$ to $\gamma$, and that the only possible accumulation points of the sequence $(U_n)_{n\in\mathbb{N}}$ in the compact space $\overline{CV_N}$ are the trees $T$ and $T'$. In view of Theorem \ref{bdy}, the sequence $(\pi(U_n))_{n\in\mathbb{N}}$ has a subsequence that converges to a point in $\partial_{\infty}FF_N$. However, the $\pi$-images of the trees $U_n$ remain in a bounded neighborhood of $\pi(S)$ in $FF_N$, which yields the desired contradiction. 
\end{proof}

Let $\mathcal{U}$ be a neighborhood of $(T,T')$ provided by Lemma \ref{X_0-compact}. Recall the definition of $\psi:\overline{CV_N}\to FF_N\cup\partial_{\infty}FF_N$ from Theorem \ref{bdy}. Given $A_0>0$ and $B\ge 0$, we let $\text{Right}_{y_0}(\mathcal{U};A_0,B)$ be the subspace of $\overline{CV_N}$ made of those trees $U\in\overline{CV_N}$ such that for all $\gamma\in\text{Tube}(\mathcal{U})$, the projection of $\psi(U)$ to the image of $\pi\circ\gamma$ in $FF_N$ is to the right, at distance at least $B$ from all trees in $\pi(\text{Tube}_{y_0}(\mathcal{U};A_0))$ (recall that $\psi(U)$ may belong to $\partial_{\infty}FF_N$; we define in this case the projection of $\psi(U)$ to the image of $\pi\circ\gamma$ as the leftmost possible accumulation point of a sequence of leftmost closest-point projections of $z_n$ to the image of $\pi\circ\gamma$, over all sequences $(z_n)_{n\in\mathbb{N}}\in FF_N^{\mathbb{N}}$ converging to $\psi(U)$).

\begin{lemma}\label{mor}
For all $A_0>0$, there exists $M>0$ such that for all sufficiently large $B>0$, one has $\overline{\text{Right}_{y_0}(\mathcal{U};A_0,B)}\subseteq \text{Right}_{y_0}(\mathcal{U};A_0,B-M)$. 
\end{lemma}

\begin{proof}
This follows from Proposition \ref{ff-geom}.
\end{proof}

We now fix a point $y_0\in CV_N$, a neighborhood $\mathcal{U}$ of $(T,T')$ provided by Lemma \ref{X_0-compact}, and we let $M$ be the maximum of the constants provided by Theorem \ref{bdy} and Lemma \ref{mor}. Given $A_0,B_0>0$, we let $X_0(A_0,B_0)$ be the subset of $CV_N$ made of those trees $U$ such that there exists $S\in\text{Tube}_{y_0}(\mathcal{U};A_0)$ satisfying $U\in\Delta(S)$, and either $U=S$, or else $U$ is fully recurrent with respect to some tree in $\text{Right}_{y_0}(\mathcal{U};A_0,B_0)$. 

\begin{lemma}\label{compact}
For all $A_0>0$, there exists $\beta>0$ such that for all $B_0>\beta$, the space $\overline{X_0(A_0,B_0)}\subseteq CV_N$ is compact.
\end{lemma}

\begin{proof}
Let $A_0>0$, and let $B_0>0$ be such that $B_0-M$ is sufficiently large (where \emph{sufficiently large} will be understood from the rest of the proof). Assume towards a contradiction that $\overline{X_0(A_0,B_0)}$ is not contained in $CV_N$. Then there exists a sequence $(U_n)_{n\in\mathbb{N}}\in {X_0(A_0,B_0)}^{\mathbb{N}}$ that leaves every compact subset of $CV_N$. Since we already know that $\overline{\text{Tube}_{y_0}(\mathcal{U};A_0)}$ is compact (Lemma \ref{X_0-compact}), we can assume that for all $n\in\mathbb{N}$, there exists a tree $U'_n\in \text{Right}_{y_0}(\mathcal{U};A_0,B_0)$ such that $U_n$ is fully recurrent with respect to $U'_n$. Since $\text{Tube}_{y_0}(\mathcal{U};A_0)$ has compact closure in $CV_N$, it meets only finitely many simplices of $CV_N$. Hence up to passing to a subsequence, we can assume that all trees $U_n$ belong to a common open simplex in $CV_N$, and that $(U_n)_{n\in\mathbb{N}}$ converges to a point $U\in\partial CV_N$. Up to passing to a further subsequence, we can also assume that the sequence $(U'_n)_{n\in\mathbb{N}}$ converges to a tree $U'\in\overline{CV_N}$. By Lemma \ref{mor}, we have $U'\in\text{Right}_{y_0}(\mathcal{U};A_0,B_0-M)$. It follows from \cite[Lemma 6.11]{BR13} that there exists a proper free factor of $F_N$ that is elliptic in both $U$ and $U'$. However if $B_0$ has been chosen so that $B_0-M$ is sufficiently large, then $\psi(U)$ and $\psi(U')$ are far apart in $FF_N$. This is a contradiction.
\end{proof}

Let $\overline{cv_0(F_N)}\subseteq\overline{cv_N}$ be the subset made of those trees $T$ such that the Lipschitz constant of an optimal map from $y_0$ to $T$ is equal to $1$. Given subsets $X\subseteq CV_N$ and $Y\subseteq \overline{CV_N}$, we denote by $\text{Opt}(X,Y)$ the set of all optimal $F_N$-equivariant maps from a tree in ${cv_0(F_N)}$ representing an element in $X$ to a tree in $\overline{cv_0(F_N)}$ representing an element in $Y$, which are linear on edges. This set is equipped with the equivariant Gromov--Hausdorff topology introduced in \cite[Section 3.2]{GL07}. If $X$ is a compact subset of $CV_N$, then the subsets of ${cv_0(F_N)}$ representing $X$ is also compact, so there is a bound on the Lipschitz constant of elements of $\text{Opt}(X,Y)$.

Let $A_0,B_0>0$ be given by Lemma \ref{compact}, and let $X_0:=X_0(A_0,B_0)$. We let $\mathcal{M}:=\text{Opt}(\overline{X_0},\text{Right}_{y_0}(\mathcal{U};A_0,B_0))$. In view of Lemma \ref{compact}, the set $\overline{\mathcal{M}}$ is compact, and Lemma \ref{mor} implies that $\overline{\mathcal{M}}\subseteq\text{Opt}(\overline{X_0},\text{Right}_{y_0}(\mathcal{U};A_0,B_0-M))$.

Given $C_0>0$ and $f\in\text{Opt}(\overline{X_0},\text{Right}_{y_0}(\mathcal{U};A_0,B_0-M))$, we let $C(f;C_0)$ be the smallest real number $C$ such that the smallest initial subsegment of $d_{CV_N}^{sym}$-diameter $C$ of the greedy folding path determined by $f$ projects to a region of $FF_N$ of diameter at least $C_0:=B_0-3M$ (here we need that $B_0$ was chosen sufficiently large so that $C_0$ is large enough). Using Theorem \ref{bdy}, we have $C(f;C_0)<+\infty$ for all $f\in\text{Opt}(\overline{X_0},\text{Right}_{y_0}(\mathcal{U};A_0,B_0-M))$ (otherwise the folding path directed by $f$ would leave the interior of outer space before having made enough progress in $FF_N$). In particular, we have $C(f;C_0)<+\infty$ for all $f\in\overline{\mathcal{M}}$. 

\begin{lemma}\label{C-bdd}
There exist $\beta_1,\beta_2>0$ such that if $C_0>\beta_1$ and $B_0>C_0+\beta_2$, then the map
\begin{displaymath}
\begin{array}{cccc}
C:&\overline{\mathcal{M}}&\to &\mathbb{R}\\
&f&\mapsto & C(f;C_0)
\end{array}
\end{displaymath}
\noindent is bounded from above. 
\end{lemma}

\begin{proof}
We proceed as in the proof of Lemma \ref{C-mod}. Let $K_1>0$ be such that $$d_{FF_N}(\pi(x),\pi(y))\le K_1 d_{CV_N}(x,y)+K_1$$ for all $x,y\in CV_N$. Assume towards a contradiction that $C(f;C_0)$ is unbounded. Then there exists a sequence $(f_n)_{n\in\mathbb{N}}\in\overline{\mathcal{M}}^{\mathbb{N}}$ such that for all $n\in\mathbb{N}$, we have $C(f_n;C_0)> n$. For all $n\in\mathbb{N}$, we denote by $\gamma_n$ the greedy folding path determined by $f_n$. Then for all $n\in\mathbb{N}$, there exists a $d_{CV_N}^{sym}$-closest point projection $x_n$ of $y_0$ on the image of $\gamma_n$, so that the smallest subsegment $\sigma_n$ of the image of $\gamma_n$ of $d_{CV_N}^{sym}$-diameter $n$ starting at $x_n$ projects to a subset of $FF_N$ of diameter at most $C_0$. Compactness of the space $\overline{\mathcal{M}}$ ensures that up to passing to a subsequence, we can assume that the maps $f_n$ converge to a map $f\in\overline{\mathcal{M}}$. The greedy folding paths $\gamma_n$ then converge uniformly on compact sets to the greedy folding path $\gamma$ determined by $f$. Thus the sequence $(x_n)_{n\in \mathbb{N}}$ accumulates to a point $x$ lying on the image of $\gamma$. Let then $z$ be to the right of $x$ on the image of $\gamma$, such that $d_{FF_N}(\pi(x),\pi(z))\ge C_0+4K_1$. There exists a sequence $(z_n)_{n\in \mathbb{N}}$ of points $z_n$ on the image of $\gamma_n$ that converges to $z$. For $n$ large enough in our subsequence, we have $d_{CV_N}^{sym}(x_n,x)<\frac{1}{2}$ and $d_{CV_N}^{sym}(z_n,z)<\frac{1}{2}$, and hence $d_{FF_N}(\pi(x_n), \pi(x)) <2K_1$ and  $d_{FF_N}(\pi(z_n), \pi(z)) <2K_1$. Therefore $d_{FF_N}(\pi(x_n), \pi(z_n))>C_0$. 
However the point $z_n$ is eventually in the segment $\sigma_n$, since its $d_{CV_N}^{sym}$-distance to $x$ is bounded, and this contradicts the fact that the diameter of the projection of $\sigma_n$ to $FF_N$ is at most $C_0$. 
\end{proof}

Given $A_0,B_0,C_0>0$ and $(T,T')\in\mathcal{UE}\times\mathcal{UE}\smallsetminus\Delta$, we let $C'_{A_0,B_0,C_0}(T,T')$ be the infimum of all real numbers $C>0$ such that all greedy folding lines $\gamma:\mathbb{R}\to CV_N$ from $T$ to $T'$ are $(C;A_0,B_0,C_0)$-progressing at all $d_{CV_N}^{sym}$-closest point projections of $y_0$ to $\gamma$. As a consequence of Lemma \ref{C-bdd}, we obtain the following fact.

\begin{cor}\label{cor-c'}
For all $A_0>0$, there exist $\beta_1,\beta_2>0$ such that for all $C_0>\beta_1$ and all $B_0>0$ such that $B_0>C_0+\beta_2$, the map $$C'_{A_0,B_0,C_0}:\mathcal{UE}\times\mathcal{UE}\smallsetminus\Delta\to\mathbb{R}$$ is locally bounded from above.
\qed
\end{cor}

\begin{cor}\label{map-C}
There exist $A_0,B_0,C_0>0$ satisfying the hypotheses from Proposition \ref{contr-out}, and an upper semi-continuous map $C:\mathcal{UE}\times\mathcal{UE}\smallsetminus\Delta \to \mathbb{R}$ such that for all $(T,T')\in\mathcal{UE}\times\mathcal{UE}\smallsetminus\Delta$, all greedy folding lines $\gamma:\mathbb{R}\to CV_N$ from $T$ to $T'$ are $(C(T,T');A_0,B_0,C_0)$-progressing at all $d_{CV_N}^{sym}$-closest-point projections of $y_0$ to $\gamma$.
\end{cor}

\begin{proof}
This follows from Lemma \ref{measurable} and Corollary \ref{cor-c'}. 
\end{proof}

We now deduce that typical geodesic rays contain infinitely many progressing (and hence contracting and high) subsegments. We fix constants $A_0,B_0,C_0>0$ provided by Corollary \ref{map-C}.

\begin{prop}\label{key-out}
Let $\mu$ be a nonelementary probability measure on $\text{Out}(F_N)$ with finite first moment with respect to $d_{CV_N}$. Then for all $\theta\in (0,1)$, there exists $C>0$ such that for $\overline{\mathbb{P}}$-a.e. bilateral sample path $\mathbf{\Phi}$ of the random walk on $(\text{Out}(F_N),\mu)$, and all greedy folding lines $\gamma:\mathbb{R}\to CV_N$ from $\text{bnd}^-(\mathbf{\Phi})$ to $\text{bnd}^+(\mathbf{\Phi})$, the set of integers $k\in\mathbb{Z}$ such that $\gamma$ is $(C;A_0,B_0,C_0)$-progressing at all $d_{CV_N}^{sym}$-closest-point projections of $y_0$, has density at least $\theta$.
\end{prop}

\begin{proof}
We recall that $\overline{\mathcal{P}}$ denotes the space of bilateral sample paths of the random walk on $(\text{Out}(F_N),\mu)$. By Corollary \ref{map-C}, the map 
\begin{displaymath}
\begin{array}{cccc}
\widetilde{C}: &\overline{\mathcal{P}}&\to &\mathbb{R}\\
&\mathbf{\Phi} &\mapsto & C(\text{bnd}^-(\mathbf{\Phi}),\text{bnd}^+(\mathbf{\Phi}))
\end{array}
\end{displaymath}
\noindent is measurable. Let $U$ be the transformation of $\overline{\mathcal{P}}$ induced by the Bernoulli shift on the space of bilateral sequences of increments. Then for $\overline{\mathbb{P}}$-a.e. bilateral sample path $\mathbf{\Phi}$ of the random walk and all $k\in\mathbb{Z}$, all greedy folding lines $\gamma:\mathbb{R}\to CV_N$ from $\text{bnd}^-(\mathbf{\Phi})$ to $\text{bnd}^+(\mathbf{\Phi})$ are $(\widetilde{C}(U^k.\mathbf{\Phi});A_0,B_0,C_0)$-progressing at all $d_{CV_N}^{sym}$-closest-point projections of $\Phi_k.y_0$ to $\gamma$. Let $\theta\in (0,1)$. We can choose $C>0$ such that $$\overline{\mathbb{P}}(\widetilde{C}(\mathbf{\Phi})\le C)>\theta.$$ As in the proof of Proposition \ref{density} in the case of mapping class groups, an application of Birkhoff's ergodic theorem to the ergodic transformation $U$ then implies that for $\mathbb{P}$-a.e. sample path $\mathbf{\Phi}:=(\Phi_n)_{n\in\mathbb{N}}$ of the random walk on $(\text{Out}(F_N),\mu)$, the density of integers $k\in\mathbb{N}$ such that $\gamma$ is $(C;A_0,B_0,C_0)$-progressing at all $d_{CV_N}^{sym}$-closest point projections of $\Phi_k.y_0$ to $\gamma$, is at least $\theta$.
\end{proof}

\begin{prop}\label{key-out-2}
Let $\mu$ be a nonelementary probability measure on $\text{Out}(F_N)$ with finite first moment with respect to $d_{CV_N}$. Then for all $\theta\in (0,1)$, there exist $B,D,M>0$ such that for $\overline{\mathbb{P}}$-a.e. bilateral sample path $\mathbf{\Phi}$ of the random walk on $(\text{Out}(F_N),\mu)$, and all greedy folding lines $\gamma:\mathbb{R}\to CV_N$ from $\text{bnd}^-(\mathbf{\Phi})$ to $\text{bnd}^+(\mathbf{\Phi})$, the set of integers $k\in\mathbb{Z}$ such that $\gamma$ is $(B,D)$-contracting and $M$-high at all $d_{CV_N}^{sym}$-closest-point projections of $y_0$, has density at least $\theta$.
\end{prop}

\begin{proof}
This follows from Propositions \ref{high-out}, \ref{contr-out} and \ref{key-out}.
\end{proof}

\subsection{End of the proof of the spectral theorem}\label{sec4-6}

\begin{proof}[Proof of Theorem \ref{theo-spectral-out}]
We will check that $(\text{Out}(F_N),\mu)$ satisfies the hypotheses in Theorem \ref{spectral-gen}. The pair $(CV_N,FF_N)$ is a hyperbolic $\text{Out}(F_N)$-electrification. All elements of $\text{Out}(F_N)$ acting loxodromically on $FF_N$ (i.e. fully irreducible outer automorphisms) have an electric translation axis in $CV_N$, and $(\text{Out}(F_N),\mu)$ is $\overline{CV_N}$-boundary converging (Theorem \ref{NPR}). The first item in the definition of contracting pencils of geodesics (Definition \ref{pencil}) was established by Namazi--Pettet--Reynolds \cite[Lemmas 7.16 and 7.17]{NPR14} for the collection of greedy folding lines between any two distinct trees in $\mathcal{UE}$. In view of Proposition \ref{key-out-2}, the second item in this definition is also satisfied by this collection of lines. Therefore, Theorem \ref{spectral-gen} applies to $(\text{Out}(F_N),\mu)$. Recall that the translation length of any fully irreducible outer automorphism on $CV_N$ is equal to the logarithm of its stretching factor. Therefore, for $\mathbb{P}$-a.e. sample path of the random walk on $(\text{Out}(F_N),\mu)$, we have $$\lim_{n\to +\infty}\frac{1}{n}\log\lambda(\Phi_n^{-1})=L,$$ where $L$ is the drift of the random walk on $(\text{Out}(F_N),\mu)$ with respect to $d_{CV_N}$. 
\end{proof}

{\small
\bibliographystyle{amsplain}
\bibliography{TS-biblio-final}

\medskip

\sc \noindent Fran\c{c}ois Dahmani,  Universit\'e de Grenoble Alpes,  Institut Fourier, F-38000 Grenoble, France.

\tt e-mail:francois.dahmani@univ-grenoble-alpes.fr

\smallskip

\sc \noindent Camille Horbez, Laboratoire de Math\'ematiques d'Orsay, Univ. Paris-Sud, CNRS, Universit\'e Paris-Saclay, 91405 Orsay, France.

\tt e-mail:camille.horbez@math.u-psud.fr

}
\end{document}